\documentclass[reqno,intlimits]{gen-j-l}

\usepackage{amssymb}

\newtheorem*{theorem*}{Theorem}

\newtheorem{theorem}{Theorem}[section]
\newtheorem{lemma}[theorem]{Lemma}

\newtheorem{proposition}[theorem]{Proposition}

\theoremstyle{definition}
\newtheorem{definition}[theorem]{Definition}
\newtheorem{example}[theorem]{Example}

\theoremstyle{remark}
\newtheorem{remark}[theorem]{Remark}

\numberwithin{equation}{section}

\usepackage{amsfonts}
\newcommand{\field}[1]{\mathbb{#1}}

\newcommand{\R}{\field{R}}

\providecommand{\abs}[1]{\lvert#1\rvert}
\providecommand{\norm}[1]{\lVert#1\rVert}

\def\acknowledgement{\par\addvspace{17pt}\small\rmfamily
\trivlist\if!\ackname!\item[]\else
\item[\hskip\labelsep
{\bfseries\ackname}]\fi}

\newenvironment{acknowledgements}{\begin{acknowledgement}}
{\end{acknowledgement}}

\def\ackname{Acknowledgements}

\begin{document}

\title[Exact multiplicity results for a singularly perturbed Neumann
problem]
{Exact multiplicity results for a singularly perturbed Neumann
problem}


\author{Massimo Grossi}
\address{Dipartimento di Matematica, Universit\`a di Roma ``La Sapienza",
P.le A. Moro 2 - 00185 Roma}
\email{\sf grossi@mat.uniroma1.it}


\author{S\'ergio L. N. Neves}
\address{Departamento de Matem‡tica, Universidade Estadual de Campinas, IMECC, Rua SŽrgio Buarque de Holanda 651, Campinas, SP, Brasil}
\email{sergio184@gmail.com}






\begin{abstract}
In this paper we study the number of the boundary single peak
solutions of the problem
\begin{align*}
\begin{cases}
  -\varepsilon^2 \Delta u + u = u^p,  &\text{ in }\Omega \\
       u > 0,                      &\text{ in }\Omega\\
       \frac{\partial u}{\partial \nu} = 0,& \text{ on }\partial \Omega
\end{cases}
\end{align*}
for $\varepsilon$ small and $p$ subcritical. \\
Under some suitable assumptions on the shape of the boundary near
a critical point of the mean curvature, we are able to prove exact multiplicity results.
Note that the degeneracy of the critical point is allowed.
\end{abstract}

\maketitle

\section{Introduction and main results}

In this paper we are concerned with the following problem
\begin{align}
\label{eq1}
\begin{cases}
  -\varepsilon^2 \Delta u + u = u^p,  &\text{ in }\Omega \\
       u > 0,                      &\text{ in }\Omega\\
       \frac{\partial u}{\partial \nu} = 0,& \text{ on }\partial \Omega
\end{cases}
\end{align}
where $\Omega \subset \R^N$ ($N \geqslant 2$) is a bounded smooth
domain, $1<p<\frac{N+2}{N-2}$ if $N > 2$  otherwise $p>1$ and
$\nu$ denotes the outer unit normal at $\partial \Omega$.

This problem was first studied in the papers \cite{LNT},
\cite{NT1}, \cite{NT2}, where the authors analyzed the asymptotic
behavior of the least energy solution of the functional naturally
associated. Among other results, they proved that for
$\varepsilon$ small enough the least energy solution to
\eqref{eq1} has exactly one local maximum and concentrates at a
point $P$ which achieves the maximum of the mean curvature of the
boundary of $\Omega$. This solution is usually called {\it single
peak solution}. In this paper we use the following definition,
\begin{definition} \label{def1}
A family of solutions $u_{\varepsilon}$ of \eqref{eq1} is called
boundary single peak if
\begin{align*}
&\varepsilon^{-N}\int_{\Omega}\left( \varepsilon^2 \abs{\nabla
u_{\varepsilon}}^2
+ u_{\varepsilon}^2 \right)dx \leqslant C \qquad \forall \,\, \varepsilon >0, \\
&u_{\varepsilon} \,\, \text{has exactly one local maximum point} \,\, P_{\varepsilon} \,\,
\text{which belongs to } \partial\Omega.
\end{align*}
Finally we say that $u_{\varepsilon}$ concentrates at $P_0$ if
$P_{\varepsilon} \to P_0$ as $\varepsilon \to 0$.
\end{definition}

After the work of Ni and Takagi there were efforts to obtain
solutions concentrating at other critical points of the mean
curvature. Some results in this direction can be found in \cite{W}
for non degenerate critical points, and in \cite{DPFW}, \cite{Gui}
and \cite{YYL} for possibly degenerate critical points.
There is an extensive literature on this subject and it would be
impossible to provide a complete list of references. For the
interested reader we refer to \cite{AM,M} and the references
therein.

In this paper we want to estimate the number of single peak
solutions of the problem \eqref{eq1} which concentrate at a given
point $P \in \partial \Omega$, without assuming that $P$ is a non
degenerate critical point of the mean curvature of $\partial
\Omega$.\\
Actually, for the non degenerate case, we have the following
result,
\begin{theorem*}[\cite{W2}]
Let us consider a non degenerate critical point $P$ of the mean
curvature. Then there exists a unique single peak solution which
concentrates at $P$.
\end{theorem*}

Therefore the question of number of single peak solutions
concentrating at the same point becomes interesting when we have a
degenerate critical point of the mean curvature.

This study reveals new phenomena which does not appear in the non
degenerate case. For example, we see that in general there is no
uniqueness of the solution. In fact, in Section 6, we give an
example where there exist two single peak solutions concentrating
at the same point.\\
In order to state the main results we need to require some
assumptions on the boundary of the domain.

\section*{Assumptions on the domain}

{\it Without loss of generality we may assume that $P=0$ is the
origin, that $x_N=0$ is the tangent plane of $\partial \Omega$ at
$0$ and $\nu(0)=(0,\dots,0, -1)$. We make the following
assumptions on the shape of $\partial \Omega$ around $0$.

There exists $r_0 > 0$ such that, in a neighborhood of $0$, $\partial \Omega$ is the graph of a function $\psi(x')$, $x' \in \R^{N-1}$ with the following properties:
\begin{equation}
\label{eq2}
\psi(x') = Q(x') + R(x'), \,\,\,\,\,\,\,\,\,\,\,\, \forall \,\,\,\, \abs{x'} <r_0
\end{equation}
where $Q$ is a smooth function defined in all $\R^{N-1}$ which satisfies, for some real number $\alpha \geqslant 3$
\begin{equation}
\label{eq3}
Q(t x')= t^{\alpha +1} Q(x'), \,\,\,\,\,\,\,\,\,\,\ \forall \,\,\, t>0 ,\,\,\,\,\forall \,\,\, x' \in \R^{N-1}
\end{equation}
and $R$ is a smooth function satisfying, for some $\beta > \alpha$ and $C > 0$
\begin{equation}
\abs{D^{\mathbf k}R(x')} \leqslant C \abs{x'}^{\beta +1 - \abs{\mathbf k}}, \,\,\,\,\,\,\,\,\, \forall \,\,\,\,\abs{x'} \leqslant r_0  \label{eq4}
\end{equation}
\noindent for all multi-index $\mathbf k = (k_1,\dots,k_N)$ with
$\abs{\mathbf k} \leqslant 4$ where $\abs{\mathbf k}=k_1 + \cdots
+ k_N$ and $D^{\mathbf k}R = \frac{\partial^{\abs{\mathbf k}}R}
{\partial x_1^{k_1}\cdots\partial x_N^{k_N}}$.}

In some cases we also consider the following condition
\begin{equation}
\label{eq97}
\nabla \Delta Q (x')\neq 0 \,\,\,\,\,\,\,\, \forall \,\,\, x' \in S^{N-2}.
\end{equation}

Note that some crucial computations in the proof of our main
results use the above parametrization of the boundary of $\Omega$.
In fact, unlike other papers in this subject where the mean
curvature plays a crucial role, here the leading term of some
important expansions involves the derivatives of the function $Q$.

Let us introduce the following vector field $\mathcal{L} :
\R^{N-1} \rightarrow \R^{N-1}$ with components
\begin{equation}
\label{eq49}
\mathcal{L}(\xi')= \left( \,\, \int_{\R^{N-1}}\left(\frac{1}{2}\abs{\nabla U}^2 + \frac{1}{2}U^2 - \frac{1}{p+1}U^{p+1} \right)(y',0)\frac{\partial Q}{\partial y_i}(y' + \xi') d y' \right)_{i=1,\dots,N-1}
\end{equation}
where $Q$ is defined in \eqref{eq2} and $U \in H^1(\R^N)$ is the
unique solution of
\begin{align}
\label{eq8}
\begin{cases}
  - \Delta U + U = U^p &\hbox{ in }\R^N \\
  U > 0&\hbox{ in }\R^N\\
  U(0)= \max\limits_{x \in \R^N} \left\{ U(x)\right\},
\end{cases}
\end{align}
which is well known to be a radial function which decays
exponentially together with its derivatives up to third order (see
\cite{BL} and \cite{K}).
Note that the integrand in \eqref{eq49} is in $L^1(\R^{N-1})$ by
the exponential decay of $U$ and $\abs{\nabla U}$. Set
\begin{equation*}
\Xi = \left\{\xi' \in \R^{N-1} \,\,\,\,\text{such that}\,\,
\xi'\,\, \text{is a stable zero of}\,\, \mathcal{L} \right\},
\end{equation*}
the definition of stable zero that we use here is the same as in \cite{G}. We recall it 
here for the reader's convenience.

\begin{definition} \label{def2}
Let $\mathcal{L} \in C(\R^{N-1},\R^{N-1})$ be a vector field. We say that $\xi_0$ is a stable zero for 
$\mathcal{L}$ if
\begin{itemize}
	\item[(i)] $\mathcal{L}(\xi_0) = 0$,
	\item[(ii)] $\xi_0$ is isolated,
	\item[(iii)] If $\mathcal{L}_n$ is a sequence of vector fields such that $\norm{\mathcal{L} - \mathcal{L}_n}_{C(B_r(\xi_0))} \to 0$, 
	for some $r>0$, then there exists $\xi_n \in B_r(\xi_0)$ such that $\mathcal{L}_n(\xi_n) = 0$ and $\xi_n \to \xi_0$.  
\end{itemize}
\end{definition}

We are now able to state our main result.

\begin{theorem} \label{mainteo}
Let $\Omega$ be a bounded domain satisfying the conditions \eqref{eq2} - \eqref{eq4}, and suppose that $\# \Xi < \infty $. Then there exists $\varepsilon_0 > 0$ such that for any $0 < \varepsilon < \varepsilon_0$
\begin{equation}
\label{eq150}
\# \left\{\text{Single peak solutions of \eqref{eq1} concentrating at} \,\, 0 \right\} \geqslant \# \Xi.
\end{equation}

If we suppose, in addition, that the domain satisfies \eqref{eq97} and that for any $\xi' \in \R^{N-1}$ such that $\mathcal{L}(\xi')=0$ it holds
\begin{equation}
\label{eq174}
\text{det Jac } \mathcal{L}(\xi') \neq 0
\end{equation}
then there exists $\varepsilon_0 > 0$ such that for any $0< \varepsilon < \varepsilon_0$ we have
\begin{equation}
\label{eq175}
\# \left\{\text{Single peak solutions of \eqref{eq1} concentrating at} \,\, 0 \right\} = \# \Xi
\end{equation}
\end{theorem}
Note that \eqref{eq175} applies also if $\Xi=\emptyset$. In this
case we have that there is no
solution concentrating at $0$ (see Proposition \ref{teo14}).\\
The proof of Theorem \ref{mainteo} relies on the classical
Lyapunov-Schmidt reduction, a tool widely used in this kind of
problems.

The paper is organized as follows: in Section \ref{s2} we
introduce the variational setting and prove some important
estimates. In Section \ref{s3} we perform the Lyapunov-Schmidt
reduction and in Section \ref{estimate} we prove \eqref{eq150} of
Theorem \ref{mainteo}. In Section \ref{s5} we recall some useful
properties of single peak solutions and prove a crucial estimate
on the rate of $P_{\varepsilon}-P_0$. In Section \ref{exact} we
give the proof of \eqref{eq175}. In Section \ref{s7} we
provide some examples and applications of Theorem \ref{mainteo}. 
Finally, we conclude with an Appendix with some technical details.

\section*{Notation}

\begin{itemize}
\item[] $x\cdot y$ denotes the scalar product of $x,y \in \R^N$;

\item[] $B_{\rho}(y) = \{ x \in \R^N \, : \, \abs{x-y} < \rho \}$;

\item[] $S^N = \{ x \in \R^{N+1} \, : \, \abs{x}=1\}$;

\item[] $Id$ denotes the identity operator;

\item[] $\nabla F$ denotes the gradient of the functional $F$;

\item[] $D^m F(x_0)$ the m-th derivative of $F$ at $x_0$;

\item[] $\# S$ denotes the cardinality of the set $S$;

\item[] $u_+ = \max\{0,u\}$;

\item[] $C$, $c_0$, $\tau_0$ denote various constants independent of $\varepsilon$;

\item[] $h=O(f(\varepsilon))$ means that $h\left(f(\varepsilon)\right)^{-1} \leqslant C$ as $\varepsilon \to 0$, where $C$ is independent of $\varepsilon$;

\item[] $h=o(f(\varepsilon))$ means that $h\left(f(\varepsilon)\right)^{-1} \to 0$ as $\varepsilon \to 0$;

\item[] $H^1(\mathcal{D}) = W^{1,2}(\mathcal{D})$, for a domain $\mathcal{D} \subseteq \R^N$ is the standard Sobolev space;

\item[] $\left( u | v \right)_{H^1(\mathcal{D})}$ denotes the inner product of $u,v
\in H^1(\mathcal{D})$.
\end{itemize}

\section{The Variational Perturbative Setting}\label{s2}
First we make a change of variables to transform the problem in $\Omega$ into a problem in $\Omega_{\varepsilon}= \frac{1}{\varepsilon} \Omega$.
\begin{align}
\label{eq6}
\begin{cases}
  - \Delta u + u = u^p, & \Omega_{\varepsilon} \\
  u>0 , & \Omega_{\varepsilon} \\
  \frac{\partial u}{\partial \nu} = 0,&\partial \Omega_{\varepsilon}
\end{cases}
\end{align}
so that if $u_{\varepsilon}(x)$ is a solution of \eqref{eq6} then $u_{\varepsilon}(\frac{x}{\varepsilon})$ is a solution of \eqref{eq1}.

Solutions of \eqref{eq6} are critical points of the functional $I_{\varepsilon} \in C^2(H^1(\Omega_{\varepsilon}), \R)$
\begin{equation}
\label{eq7}
I_{\varepsilon}(u) = \frac{1}{2}\int_{\Omega_{\varepsilon}}{(\abs{\nabla u}^2 + u^2)} - \frac{1}{p+1} \int_{\Omega_{\varepsilon}}{u_+^{p+1}}.
\end{equation}

Next we construct a manifold $\mathcal{Z}^{\varepsilon} \subseteq H^1(\Omega_{\varepsilon})$, of class
$C^2$, consisting of pseudo-critical
points for $I_{\varepsilon}$ in the sense that
$\norm{I_{\varepsilon}'(z)}$ is small for all $z \in
\mathcal{Z}^{\varepsilon}$.

Let $U \in H^1(\R^N)$ be the  solution of \eqref{eq8} and let
$\mathcal{Z}^{\varepsilon}$ be defined by
\begin{equation}
\label{eq9}
\mathcal{Z^{\varepsilon}} = \left\{U_{\xi}=U(\cdot - \xi) \, | \, \xi \in \partial \Omega_{\varepsilon} \right\}.
\end{equation}
The tangent space to $\mathcal{Z}^{\varepsilon}$ at $U_{\xi}$ is denoted by $T_{U_{\xi}}\mathcal{Z^{\varepsilon}}$ and can be written as
\begin{equation}
\label{eq10}
T_{U_{\xi}}\mathcal{Z^{\varepsilon}} = \text{Span}_{H^1(\Omega_{\varepsilon})} \left\{\partial_{\xi_1}U_{\xi}, \dots, \partial_{\xi_{N-1}}U_{\xi}  \right\}.
\end{equation}
Hereafter $\partial_{\xi_i}$ denotes $\frac{\partial}{\partial
e_i}$ and $\left\{e_1, \dots, e_{N-1} \right\}$ are $N-1$ linearly
independent tangent vectors to $\partial \Omega_{\varepsilon}$ at
$\xi$.
\begin{lemma}
\label{le3}
Given $R > 0$ there exists $\varepsilon_0 >0$ and $C > 0$ such that for all ${0 < \varepsilon < \varepsilon_0}$ and for all $\xi \in \partial \Omega_{\varepsilon}$, $\xi = (\xi', \frac{1}{\varepsilon}\psi(\varepsilon \xi'))$ with $\abs{\xi'} \leqslant R$ we have
\begin{align}
\norm{I_{\varepsilon}'(U_{\xi})} &\leqslant C \varepsilon^{\alpha} \label{eq11}\\
\norm{I_{\varepsilon}''(U_{\xi})[q]} &\leqslant C \varepsilon^{\alpha} \norm{q} \label{eq12}
\end{align}
for every $q \in T_{U_{\xi}}\mathcal{Z^{\varepsilon}}$, where
$\alpha$ is the same number appearing in \eqref{eq3}.
\end{lemma}
\begin{proof}
Before starting the proof, we remark that this kind of result can
be found in \cite[Proposition 18]{M}. The difference is that,
since our domain is ``flatter'' we obtain estimates of order
$\varepsilon^{\alpha}$. Moreover, for our purposes, it is
important that the constant $C$ is uniform for $\xi'$ in a fixed
ball.

First we prove \eqref{eq11}. Integrating by parts and using \eqref{eq8} we can write, for any $v \in H^1(\Omega_{\varepsilon})$,
\begin{equation*}
\left( I_{\varepsilon}'(U_{\xi})| v \right) = \int_{\partial
\Omega_{\varepsilon}}{\frac{\partial U_{\xi}}{\partial \nu} v
d\sigma}.
\end{equation*}

Now we divide $\partial \Omega_{\varepsilon}$ into two parts, $\partial \Omega_{\varepsilon} \cap B_{\frac{r_0}{\varepsilon}}(0)$ and $\partial \Omega_{\varepsilon} \backslash B_{\frac{r_0}{\varepsilon}}(0)$. In the latter set we use the exponential decay of $U_{\xi}$ and its derivatives, plus the trace inequality (with a constant independent of $\varepsilon$) to obtain
\begin{equation}
\label{eq13}
\left|\,\,\,\int_{\partial \Omega_{\varepsilon} \backslash B_{\frac{r_0}{\varepsilon}}(0)}{\frac{\partial U_{\xi}}{\partial \nu} v d \sigma}\right| \leqslant C \text{e}^{- \frac{c_0}{\varepsilon}} \norm{v}_{H^1(\Omega_{\varepsilon})}
\end{equation}
for some $c_0 > 0$.
In $ B_{\frac{r_0}{\varepsilon}}(0)$ we use the function $\psi$ to parametrize
\begin{equation*}
\partial \Omega_{\varepsilon} \cap  B_{\frac{r_0}{\varepsilon}}(0) \subseteq \left\{\left(y', \frac{1}{\varepsilon}\psi(\varepsilon y')\right) \, \Big| \, \abs{y'} < \frac{r_0}{\varepsilon} \right\}
\end{equation*}
 and the fact that $\nabla U(x) = U'(\abs{x})\frac{x}{\abs{x}}$ to write
\begin{equation*}
\frac{\partial U_{\xi}}{\partial \nu} (y) = \frac{U'(\abs{y - \xi})}{\abs{y-\xi}} (y - \xi) \cdot \nu(y), \,\,\,\,\,\,\,y=(y', \frac{1}{\varepsilon}\psi(\varepsilon y'))
\end{equation*}
where
\begin{equation*}
\nu(y) = \frac{(\nabla \psi (\varepsilon y'), -1)}{\sqrt{1+ \abs{\nabla \psi (\varepsilon y')}^2}}\, .
\end{equation*}
So, by \eqref{eq3}, \eqref{eq4} and the Mean Value Theorem, one finds that
\begin{align}
\left| \frac{\partial U_{\xi}}{\partial \nu} (y)\right|& = \left|\frac{U'(\abs{y - \xi})}{\abs{y - \xi}\sqrt{1+ \abs{\nabla \psi (\varepsilon y')}^2}} \left( y' - \xi', \frac{1}{\varepsilon}\psi(\varepsilon y') - \frac{1}{\varepsilon}\psi(\varepsilon \xi')\right) \cdot \left(\nabla \psi(\varepsilon y'), -1\right)\right| \notag\\
& \leqslant \frac{C \text{e}^{ -\lambda \abs{y' - \xi'}}}{\abs{y' - \xi'}} \left(\abs{y' - \xi'} \abs{\nabla \psi(\varepsilon y')} + \frac{1}{\varepsilon}\abs{\psi(\varepsilon y') - \psi(\varepsilon \xi')}\right) \notag\\
& \leqslant C \text{e}^{- \lambda \abs{y' - \xi'}} \left(\varepsilon^{\alpha} \abs{y'}^{\alpha} + \varepsilon^{\beta }\abs{y'}^{\beta} + \left|\nabla \psi(\varepsilon(y' + \theta(\xi' - y')))\right|\right) \,\,\,\, \theta \in (0,1)\notag\\
& \leqslant C \text{e}^{- \lambda \abs{y' - \xi'}} \varepsilon^{\alpha}\left(\abs{y'}^{\alpha} + \abs{y'}^{\beta} + \abs{y' - \xi'}^{\alpha} + \abs{y' - \xi'}^{\beta}\right) \notag\\
& \leqslant C \varepsilon^{\alpha} \text{e}^{\lambda R}\text{e}^{- \lambda \abs{y'}}\left(\abs{y'}^{\alpha} + \abs{y'}^{\beta} + R^{\beta}\right) \leqslant  C \varepsilon^{\alpha} \text{e}^{- \frac{\lambda}{2} \abs{y'}} \label{eq14}
\end{align}
The above estimate and the trace inequality yield
\begin{equation}
\label{eq15}
\left|\,\,\,\int_{\partial \Omega_{\varepsilon} \cap B_{\frac{r_0}{\varepsilon}}(0)}{\frac{\partial U_{\xi}}{\partial \nu} v d \sigma}\right| \leqslant C \varepsilon^{\alpha} \norm{v}_{H^1(\Omega_{\varepsilon})}.
\end{equation}
By \eqref{eq13} and \eqref{eq15} we obtain \eqref{eq11}.

Let us now prove \eqref{eq12}. We take $e_i=\left(0, \dots, 1, \dots, 0, \frac{\partial \psi}{\partial y_i}(\varepsilon \xi') \right)$, $i = 1, \dots, N-1$ as a basis of the tangent space to $\partial \Omega_{\varepsilon}$ at $\xi = (\xi', \frac{1}{\varepsilon} \psi(\varepsilon \xi'))$. In this form, the directional derivatives $\frac{\partial U_{\xi}}{\partial e_i}$ are given by
\begin{equation*}
\frac{\partial U_{\xi}}{\partial e_i}(y) = \frac{\partial U_{\xi}}{\partial y_i}(y) + \frac{\partial U_{\xi}}{\partial y_N}(y)\frac{\partial \psi}{\partial y_i}(\varepsilon \xi')
\end{equation*}
and then, using again \eqref{eq2},
\begin{equation}
\label{eq16}
\left\|\frac{\partial U_{\xi}}{\partial e_i} - \frac{\partial U_{\xi}}{\partial y_i} \right\|_{H^1(\Omega_{\varepsilon})} = \,\, O(\varepsilon^{\alpha}).
\end{equation}
Recall that $O(\varepsilon^{\alpha})$, as $\varepsilon
\rightarrow 0$, is uniform in $\xi'$ since $\abs{\xi'} \leqslant
R$. We claim that
\begin{equation}
\label{eq17} \left(\frac{\partial U_{\xi}}{\partial e_i}
\left|\frac{\partial U_{\xi}}{\partial e_j}\right.
\right)_{H^1(\Omega_{\varepsilon})} = C_0 \delta_{ij} +
O(\varepsilon^{\alpha}) \,\,\,\,\,\,\,\,i,j=1, \dots, N-1,
\end{equation}
where $\delta_{ij} =  \begin{cases}
                              1, i=j\\
                              0, i \neq j
                              \end{cases}$
and $C_0 > 0$.

By \eqref{eq17} it suffices to prove \eqref{eq12} only for $q =
\frac{q_i}{\norm{q_i}_{H^1(\Omega_{\varepsilon})}}$, $q_i=
\frac{\partial U_{\xi}}{\partial e_i}$. Let ${v \in
H^1(\Omega_{\varepsilon})}$ then
\begin{align}
& \left|\left(I_{\varepsilon}''(U_{\xi})\left.\left[\frac{q_i}{\norm{q_i}_{H^1(\Omega_{\varepsilon})}} \right]  \right| v   \right)  \right|= \left| \frac{1}{\norm{q_i}} \left[ \left(\left.\frac{\partial U_{\xi}}{\partial e_i} \right| v \right)_{H^1(\Omega_{\varepsilon})}  - p \int_{\Omega_{\varepsilon}}{U_{\xi}^{p-1} \frac{\partial U_{\xi}}{\partial e_i} v}\right]\right| \notag\\
& \leqslant \left| \,\, \int_{\Omega_{\varepsilon}}{\left(\nabla \frac{\partial U_{\xi}}{\partial y_i} \cdot \nabla v + \frac{\partial U_{\xi}}{\partial y_i}v - p U_{\xi}^{p-1} \frac{\partial U_{\xi}}{\partial y_i} v\right) dy } \right| + O(\varepsilon^{\alpha}) \norm{v} \notag\\
& \leqslant \left|\,\, \int_{\Omega_{\varepsilon}}{\left(- \Delta \frac{\partial U_{\xi}}{\partial y_i}  + \frac{\partial U_{\xi}}{\partial y_i} - p U_{\xi}^{p-1} \frac{\partial U_{\xi}}{\partial y_i} \right)v\,\, dy } + \int_{\partial \Omega_{\varepsilon}}{ \frac{\partial}{\partial \nu} \frac{\partial U_{\xi}}{\partial y_i}v d \sigma }\right| +  O(\varepsilon^{\alpha}) \norm{v} \notag\\
& \leqslant \left(\,\,\int_{\partial \Omega_{\varepsilon}}{ \left| \frac{\partial}{\partial \nu} \frac{\partial U_{\xi}}{\partial y_i} \right|^2  d \sigma } \right)^{\frac{1}{2}} \norm{v} + O(\varepsilon^{\alpha}) \norm{v} \leqslant C \varepsilon^{\alpha} \norm{v} \label{eq18}
\end{align}
where we again used the exponential decay of the derivatives of $U$ just as in \eqref{eq13} and \eqref{eq15} (see Appendix \ref{apc}).

It remains to prove \eqref{eq17} which is a straightforward calculation. Using the Mean Value Theorem and the exponential decay of $U$ one can prove that, for $i=1, \dots, N-1$
\begin{equation}
\label{eq5}
\left\|\frac{\partial U_{\xi}}{\partial y_i} - \frac{\partial U_{\xi'}}{\partial y_i} \right\|_{H^1(\Omega_{\varepsilon})} = \,\, O(\varepsilon^{\alpha})
\end{equation}
where ${U_{\xi'} = U_{(\xi',0)}}$ (see Appendix \ref{apd}). Thus by \eqref{eq16} and \eqref{eq5} it suffices to prove
\begin{equation}
\label{eq19}
\left(\frac{\partial U_{\xi'}}{\partial y_i} \left| \frac{\partial U_{\xi'}}{\partial y_j} \right. \right)_{H^1(\Omega_{\varepsilon})}= \, \,C_0\, \delta_{ij} + O(\varepsilon^{\alpha}), \,\,\,\,\,\,\,\, i,j=1, \dots, N-1.
\end{equation}
In order to do this we write
\begin{align}
&\left(\frac{\partial U_{\xi'}}{\partial y_i} \left| \frac{\partial U_{\xi'}}{\partial y_j} \right. \right)_{H^1(\Omega_{\varepsilon})} = \int_{\R_+^N}{\left(\nabla \frac{\partial U_{\xi'}}{\partial y_i} \cdot \nabla \frac{\partial U_{\xi'}}{\partial y_j} + \frac{\partial U_{\xi'}}{\partial y_i} \frac{\partial U_{\xi'}}{\partial y_j}\right)}  \label{eq20} \\
&- \int_{\R_+^N \backslash \Omega_{\varepsilon}}{\left(\nabla \frac{\partial U_{\xi'}}{\partial y_i} \cdot \nabla \frac{\partial U_{\xi'}}{\partial y_j} + \frac{\partial U_{\xi'}}{\partial y_i} \frac{\partial U_{\xi'}}{\partial y_j}\right)}
+  \int_{\Omega_{\varepsilon} \backslash \R_+^N}{\left(\nabla \frac{\partial U_{\xi'}}{\partial y_i} \cdot \nabla \frac{\partial U_{\xi'}}{\partial y_j} + \frac{\partial U_{\xi'}}{\partial y_i} \frac{\partial U_{\xi'}}{\partial y_j}\right)}. \notag
\end{align}
After a change of variables the first integral on the right-hand side of \eqref{eq20} can be written as
\begin{equation*}
\int_{\R_+^N}{\left(\nabla \frac{\partial U}{\partial y_i} \cdot \nabla \frac{\partial U}{\partial y_j} + \frac{\partial U}{\partial y_i} \frac{\partial U}{\partial y_j}\right)} = \left( \frac{1}{2N} \norm{U}_{H^1(\R^N)}^2 \right)\delta_{ij} = C_0\, \delta_{ij}
\end{equation*}

Now we estimate the other integrals on the right-hand side of \eqref{eq20}. We will estimate only the integral which involves the first order derivative, the other is completely analogous. Away from $0$ their values are exponentially small in $\varepsilon$ so it remains to estimate them in $A_{\frac{r_0}{\varepsilon}} = \left\{\left(y', y_N\right)\, \Big| \,\,\,\abs{y'} < \frac{r_0}{\varepsilon}\,\, ,\,\, \abs{y_N} < \frac{r_0}{\varepsilon} \right\}$
\begin{align*}
& \left|\,\, - \int_{(\R_+^N \backslash \Omega_{\varepsilon}) \cap A_{\frac{r_0}{\varepsilon}}}{ \frac{\partial U_{\xi'}}{\partial y_i} \frac{\partial U_{\xi'}}{\partial y_j}} \,\,\, + \int_{(\Omega_{\varepsilon} \backslash \R_+^N) \cap A_{\frac{r_0}{\varepsilon}}}{ \frac{\partial U_{\xi'}}{\partial y_i} \frac{\partial U_{\xi'}}{\partial y_j}} \,\,\, \right| \\
& = \left|- \,\, \int_{\abs{y'} < \frac{r_0}{\varepsilon}} \int_0^{\frac{1}{\varepsilon}\psi(\varepsilon y')}{\frac{\partial U_{\xi'}}{\partial y_i}(y',y_N)\frac{\partial U_{\xi'}}{\partial y_j}(y', y_N) d y_N d y'}   \right|\\
& \leqslant C \,\,\int_{\abs{y'} < \frac{r_0}{\varepsilon}}\int_0^{\abs{\frac{1}{\varepsilon}\psi(\varepsilon y')}}{\exp{\left(-\lambda (\abs{y'-\xi'} + \abs{y_N})\right)} d y_N d y'}  \\
& \leqslant C \,\,\int_{\abs{y'} < \frac{r_0}{\varepsilon}}  \int_0^{\abs{\frac{1}{\varepsilon}\psi(\varepsilon y')}}{\text{e}^{-\lambda\abs{y'-\xi'}} \text{e}^{-\lambda \abs{y_N}} d y_N d y'} \\
& \leqslant C \,\,\int_{\abs{y'} < \frac{r_0}{\varepsilon}} \text{e}^{-\lambda\abs{y'-\xi'}} \int_0^{\abs{\frac{1}{\varepsilon}\psi(\varepsilon y')}}{ \text{e}^{-\lambda \abs{y_N}} d y_N d y'} \\
& \leqslant C \,\,\int_{\abs{y'} < \frac{r_0}{\varepsilon}} \text{e}^{-\lambda\abs{y'-\xi'}} \left|\frac{1}{\varepsilon}\psi(\varepsilon y')\right| d y' \\
& \leqslant C \epsilon^{\alpha} \text{e}^{\lambda R} \,\,\int_{\R^{N-1}} \text{e}^{-\lambda\abs{y'}} (\abs{y'}^{\alpha +1} + \abs{y'}^{\beta +1}) d y' \\
& \leqslant C \epsilon^{\alpha}.
\end{align*}
\end{proof}
\section{The Lyapunov-Schmidt Reduction}\label{s3}
As in \cite{AM,M} we look for critical points of $I_{\varepsilon}$ in the form $u = z+w$ with $z \in \mathcal{Z}^{\varepsilon}$ and $w \in W = (T_z \mathcal{Z}^{\varepsilon})^\bot$. If $\mathcal{P} : H^1(\Omega_{\varepsilon}) \rightarrow W$ denotes the orthogonal projection onto $W$, the equation $I_{\varepsilon}'(z+w) =0$ is clearly equivalent to the following system
\begin{align}
\label{eq21}
\begin{cases}
  \mathcal{P} I_{\varepsilon}'(z+w)=0, \\
  (Id - \mathcal{P})I_{\varepsilon}'(z+w)=0.
\end{cases}
\end{align}
We decompose
\begin{equation*}
W= \left(T_{U_{\xi}} \mathcal{Z}^{\varepsilon} \right)^{\bot} = W_1 \oplus W_2 \, , \,\,\,\,\,\,W_1 \,\bot\, W_2
\end{equation*}
where $W_1= \text{Span}_{H^1(\Omega_{\varepsilon})}\left\{ \mathcal{P} U_{\xi}\right\}$ and $W_2=\left( \text{Span}_{H^1(\Omega_{\varepsilon})}\left\{U_{\xi}, \frac{\partial U_{\xi}}{\partial e_1}, \dots, \frac{\partial U_{\xi}}{\partial e_{N-1}} \right\}\right)^{\bot}$.

The following two Lemmas are well known see \cite[Proposition 1.2]{YYL}, see also \cite[Proposition 19]{M}.

\begin{lemma}
\label{le4}
Given $R>0$ there exists $C>0$ and $\varepsilon_0 >0$ such that
\begin{equation}
\label{eq22}
I_{\varepsilon}''(U_{\xi})[v,v] \geqslant C \norm{v}^2, \,\,\,\,\,\,\,\,\,\,\,\,\, \text{for every}\,\,v \in W_2
\end{equation}
for all $\xi=(\xi', \frac{1}{\varepsilon}\psi(\varepsilon \xi'))$, $\abs{\xi'} \leqslant R$ and $0< \varepsilon< \varepsilon_0$.
\end{lemma}

\begin{lemma}
\label{le5}
Given $R>0$ there exists $C > 0$ and $\varepsilon_0 > 0$ such that
\begin{equation}
\label{eq38}
\left|I_{\varepsilon}''(U_{\xi})[\mathcal{P}U_{\xi},\mathcal{P}U_{\xi}]\right| \geqslant C \norm{\mathcal{P}U_{\xi}}^2
\end{equation}
for all $\xi=(\xi', \frac{1}{\varepsilon}\psi(\varepsilon \xi'))$, $\abs{\xi'} \leqslant R$ and $0< \varepsilon< \varepsilon_0$.
\end{lemma}

Consider the operator
$$L_{\varepsilon, \xi}=
\mathcal{P}I_{\varepsilon}''(U_{\xi})|_W : W \rightarrow W.$$ As a
consequence of Lemmas \ref{le4} and \ref{le5} one can prove
\begin{equation}
\label{eq44}
\norm{L_{\varepsilon, \xi}(v)} \geqslant C \norm{v} \,\,\,\,\,\,\,\,\, \forall \,\ v \in W.
\end{equation}
Therefore, for $\varepsilon$ small, $L_{\varepsilon, \xi}$ is invertible  and
\begin{equation}
\label{eq46}
\norm{L_{\varepsilon, \xi}^{-1}} \leqslant \frac{1}{C}.
\end{equation}

This property allows us to perform a finite-dimensional reduction of problem \eqref{eq6} on the manifold $\mathcal{Z}^{\varepsilon}$.

\begin{proposition}
\label{prop6} Given $R>0$ there exists $\varepsilon_0 >0$ and $C>
0$ such that for all ${\xi=(\xi',
\frac{1}{\varepsilon}\psi(\varepsilon \xi')) \in \partial
\Omega_{\varepsilon}}$ with $\abs{\xi'} \leqslant R$ and $0 <
\varepsilon < \varepsilon_0$ there exists a unique
${w=w(\varepsilon, \xi) \in W}$ such that $I_{\varepsilon}'(U_\xi
+ w ) \in T_{U_{\xi}} \mathcal{Z}^{\varepsilon}$. Moreover the
function $w(\varepsilon, \xi)$ is of class $C^1$ with respect to
$\xi \in
\partial \Omega_{\varepsilon}$ and satisfies
\begin{align}
\norm{w} \leqslant &C \varepsilon^{\alpha} \label{eq47}\\
\norm{\partial_{\xi_i}w} \leqslant& C \varepsilon^{\gamma \alpha}
\,\,\,\,\,\,\,\,\text{for} \,\,\, i= 1, \dots, N-1 \label{eq48}
\end{align}
\noindent where   $\gamma= \text{min} \left\{1,p-1\right\}$.
Moreover, the function ${\Phi_{\varepsilon} : B_R(0) \cap \partial \Omega_{\varepsilon} \rightarrow \R}$ defined by $\Phi_{\varepsilon}(\xi) = I_{\varepsilon}(U_{\xi}+ w(\varepsilon, \xi))$ is of class $C^1$ and
\begin{equation*}
\nabla \Phi_{\varepsilon}(\xi_0) = 0 \,\,\,\,\, \Longrightarrow \,\,\,\, I_{\varepsilon}'(U_{\xi_0}+ w(\varepsilon, \xi_0))=0
\end{equation*}
\end{proposition}
\begin{proof}
This is well known, for example we refer to \cite[Proposition 2.1]{AMMP}.
\end{proof}

\section{Estimate of the number of critical points} \label{estimate}

In this section we give the proof of the first part of Theorem \ref{mainteo}.

\begin{proof}[Proof of \eqref{eq150}]
We take $R>0$ such that $\Xi \subseteq B_R(0) \subset \R^{N-1}$
and $\varepsilon_0 > 0$ such that all the results of the previous
section hold true and $R < \frac{r_0}{\varepsilon}$ for $0 <
\varepsilon < \varepsilon_0$.

By Proposition \ref{prop6} it suffices to estimate the number of critical points of the functional ${\Phi_{\varepsilon} : B_R(0) \cap \partial \Omega_{\varepsilon} \rightarrow \R}$ defined by
$\Phi_{\varepsilon}(\xi)=I_{\varepsilon}(U_\xi + w(\varepsilon, \xi))$. One has
\begin{equation}
\label{eq51}
\partial_{\xi_i}\Phi_{\varepsilon}(\xi)= \left( I_{\varepsilon}'(U_\xi + w)| \partial_{\xi_i}U_{\xi}+ \partial_{\xi_i}w \right).
\end{equation}
Next we write
\begin{equation}
\label{eq52}
I_{\varepsilon}'(U_\xi + w) = I_{\varepsilon}'(U_\xi) + I_{\varepsilon}''(U_\xi)[w] + R(\xi, w)
\end{equation}
where
\begin{equation*}
R(\xi, w) = I_{\varepsilon}'(U_\xi + w) - I_{\varepsilon}'(U_\xi) - I_{\varepsilon}''(U_\xi)[w]
\end{equation*}
and
\begin{equation*}
R(\xi, w)[v] = - \int_{\Omega_{\varepsilon}} \left[ (U_\xi + w)_{+}^p - U_\xi^p - pU_\xi^{p-1}w \right]v.
\end{equation*}
By the inequality
\begin{equation*}
|(a+b)_+^p - a_+^p - pa_+^{p-1}b| \leqslant
\begin{cases}
C|b|^p  &\text{for} \,\,\, p \leqslant 2 \\
C(|b|^2 + |b|^p)  &\text{for} \,\,\, p > 2
\end{cases}
\end{equation*}
for all $a,b \in \R$ such that $|a| \leqslant M$ with constant $C = C(p,M)$ we have
\begin{equation}
\label{eq52a} \norm{R(\xi,w)} \leqslant C(\norm{w}^2 +
\norm{w}^p).
\end{equation}
Hence
\begin{align}
\partial_{\xi_i}\Phi_{\varepsilon}(\xi)& = \left(I_{\varepsilon}'(U_\xi)+I_{\varepsilon}''(U_\xi
[w]| \partial_{\xi_i}U_{\xi}+ \partial_{\xi_i}w \right) + O(\varepsilon^{(1+ \gamma)\alpha})
\quad \text{\small by \eqref{eq52}-\eqref{eq52a} and \eqref{eq47},\eqref{eq48}}, \notag \\
& = \left(I_{\varepsilon}'(U_\xi)| \partial_{\xi_i}U_{\xi} + \partial_{\xi_i}w \right) + \left(I_{\varepsilon}''(U_\xi)[w]| \partial_{\xi_i}U_{\xi} +  \partial_{\xi_i}w \right) + O(\varepsilon^{(1+ \gamma)\alpha})  \notag \\
& = \left(I_{\varepsilon}'(U_\xi)| \partial_{\xi_i}U_{\xi} + \partial_{\xi_i}w \right) + \left(I_{\varepsilon}''(U_\xi)[w]| \partial_{\xi_i}U_{\xi}\right) + O(\varepsilon^{(1+ \gamma)\alpha})  \, \text{\small by \eqref{eq47}-\eqref{eq48} } \notag \\
& = \left(I_{\varepsilon}'(U_\xi)| \partial_{\xi_i}U_{\xi} + \partial_{\xi_i}w \right) + O(\varepsilon^{(1+ \gamma)\alpha})  \qquad \text{\small by \eqref{eq12} } \notag \\
& = \left(I_{\varepsilon}'(U_\xi)| \partial_{\xi_i}U_{\xi} \right) + O(\varepsilon^{(1+ \gamma)\alpha})  \qquad \text{\small by \eqref{eq11} and \eqref{eq48} } \notag \\
& = \left( I_{\varepsilon}'(U_\xi) \left| \frac{\partial U_{\xi}}{\partial y_i} \right. \right) + O(\varepsilon^{(1+ \gamma)\alpha}) \qquad \text{\small by \eqref{eq11} and \eqref{eq16} .} \label{eq53}
\end{align}
As in the proof of Lemma \ref{le3} we write
\begin{align}
\left( I_{\varepsilon}'(U_\xi) \left| \frac{\partial U_{\xi}}{\partial y_i} \right. \right) &= \int_{\partial \Omega_{\varepsilon}}{\frac{\partial U_{\xi}}{\partial \nu}\frac{\partial U_{\xi}}{\partial y_i} \,d \sigma} \notag \\
& = \int_{G_{\frac{r_0}{\varepsilon}} \cap \partial \Omega_{\varepsilon}}{\frac{\partial U_{\xi}}{\partial \nu}\frac{\partial U_{\xi}}{\partial y_i} \,d \sigma} + O \left(\text{e}^{- \frac{c_0}{\varepsilon}}\right) \label{eq54}
\end{align}
and evaluate the above integral using the function $\psi$ to parametrize the portion of $\partial \Omega_{\varepsilon}$ in $G_{\frac{r_0}{\varepsilon}} = \left\{(y', \frac{1}{\varepsilon}\psi(\varepsilon y'))\, \left|\right. \,\abs{y'}< \frac{r_0}{\varepsilon} \right\}$
\begin{align}
& \int_{G_{\frac{r_0}{\varepsilon}} \cap \partial \Omega_{\varepsilon}}{\frac{\partial U_{\xi}}{\partial \nu}(y)\frac{\partial U_{\xi}}{\partial y_i}(y) \,d \sigma_y} \notag \\
& = \int_{\abs{y'}< \frac{r_0}{\varepsilon}}{\nabla U_{\xi}\left(y', \frac{1}{\varepsilon}\psi(\varepsilon y')\right) \cdot (\nabla \psi(\varepsilon y'), -1) \frac{\partial U_{\xi}}{\partial y_i}\left(y', \frac{1}{\varepsilon}\psi(\varepsilon y')\right) d y'}  \notag \\
\begin{split}
& = \int_{\abs{y'}< \frac{r_0}{\varepsilon}} \Bigg[\sum_{j=1}^{N-1}\frac{\partial U}{\partial y_j}\left(y'- \xi',\frac{1}{\varepsilon}\psi(\varepsilon y') - \frac{1}{\varepsilon}\psi(\varepsilon \xi')\right)\frac{\partial \psi}{\partial y_j}(\varepsilon y')      \\
& - \frac{\partial U}{\partial y_N}\left(y' - \xi', \frac{1}{\varepsilon}\psi(\varepsilon y')-\frac{1}{\varepsilon}\psi(\varepsilon \xi')\right) \Bigg] \frac{\partial U}{\partial y_i}\left(y'- \xi',\frac{1}{\varepsilon}\psi(\varepsilon y') - \frac{1}{\varepsilon}\psi(\varepsilon \xi')\right) dy' .
\end{split} \label{eq55}
\end{align}
By Taylor's Theorem we can write
\begin{align*}
\frac{\partial U}{\partial y_j}\left(y'- \xi',\frac{1}{\varepsilon}\psi(\varepsilon y') - \frac{1}{\varepsilon}\psi(\varepsilon \xi')\right) &= \frac{\partial U}{\partial y_j}\left(y'- \xi',0\right) + O(\varepsilon^{\alpha}), \,\,\,\ j=1, \dots, N-1 \\
\frac{\partial U}{\partial y_N}\left(y'- \xi',\frac{1}{\varepsilon}\psi(\varepsilon y') - \frac{1}{\varepsilon}\psi(\varepsilon \xi')\right) &= \frac{\partial U}{\partial y_N}\left(y'- \xi',0\right)\\
+&\frac{\partial^2U}{\partial y_N^2}(y'-\xi',0)\left(\frac{1}{\varepsilon}\psi(\varepsilon y')- \frac{1}{\varepsilon}\psi(\varepsilon \xi')\right) +  O(\varepsilon^{2 \alpha})
\end{align*}
By \eqref{eq53}, \eqref{eq54}, \eqref{eq55} and the expansions above we infer
\begin{equation}
\begin{split}
\partial_{\xi_i} \Phi_{\varepsilon}(\xi) &= \int_{\abs{y'}< \frac{r_0}{\varepsilon}} \frac{\partial U}{\partial y_i}\left(y'- \xi',0\right) \Bigg( \sum_{j=1}^{N-1}\frac{\partial U}{\partial y_j}\left(y'- \xi',0\right)\frac{\partial \psi}{\partial y_j}(\varepsilon y')  \\
&  - \frac{\partial^2 U}{\partial y_N^2}\left(y' - \xi', 0\right)\frac{1}{\varepsilon}\psi(\varepsilon y') \Bigg) dy'
 + O(\varepsilon^{(1+\gamma)\alpha}).
\end{split} \label{eq56}
\end{equation}

By \eqref{eq2}, \eqref{eq3} and \eqref{eq4} we have
\begin{align*}
\begin{split}
\partial_{\xi_i} \Phi_{\varepsilon}(\xi) &= \varepsilon^{\alpha} \Bigg[\, \int_{\R^{N-1}} \frac{\partial U}{\partial y_i}\left(y'- \xi',0\right)\Bigg(\sum_{j=1}^{N-1}\frac{\partial U}{\partial y_j}\left(y'- \xi',0\right)\frac{\partial Q}{\partial y_j}(y')  \\
& - \frac{\partial^2 U}{\partial y_N^2}(y' - \xi', 0)Q(y') \Bigg) dy'
 +\,\, O(\varepsilon^{\beta - \alpha}) + O(\varepsilon^{\gamma \alpha}) \Bigg]
\end{split}
\\
\begin{split}
&= \varepsilon^{\alpha} \Bigg[\, \int_{\R^{N-1}} \sum_{j=1}^{N-1} \frac{\partial U}{\partial y_i}\left(y'- \xi',0\right)\frac{\partial U}{\partial y_j}\left(y'- \xi',0\right)\frac{\partial Q}{\partial y_j}(y') dy' \\
& - \int_{\R^{N-1}} \frac{\partial U}{\partial y_i}\left(y'- \xi',0\right) \frac{\partial^2 U}{\partial y_N^2}(y' - \xi', 0)Q(y') dy'
 +\,\, O(\varepsilon^{\beta - \alpha}) + O(\varepsilon^{\gamma \alpha}) \Bigg].
\end{split}
\end{align*}
Integrating by parts we get
\begin{align*}
\begin{split}
\partial_{\xi_i} \Phi_{\varepsilon}(\xi) = \varepsilon^{\alpha} &\Bigg[\, -\int_{\R^{N-1}} \sum_{j=1}^{N-1} \left( \frac{\partial^2 U}{\partial y_i\partial y_j}\frac{\partial U}{\partial y_j} + \frac{\partial U}{\partial y_i}\frac{\partial^2 U}{\partial y_j^2}\right)\left(y'- \xi',0\right) Q(y') dy' \\
& - \int_{\R^{N-1}} \left(\frac{\partial U}{\partial y_i} \frac{\partial^2 U}{\partial y_N^2}\right)(y' - \xi', 0)Q(y') dy'
 +\,\, O(\varepsilon^{\beta - \alpha}) + O(\varepsilon^{\gamma \alpha}) \Bigg]
\end{split} \\
\begin{split}
= \varepsilon^{\alpha} &\Bigg[\, -\int_{\R^{N-1}}\left( \frac{1}{2}\frac{\partial}{\partial y_i}\abs{\nabla U}^2 + \Delta U \frac{\partial U}{\partial y_i}\right)\left(y'- \xi',0\right) Q(y') dy' \\
&  +\,\, O(\varepsilon^{\beta - \alpha}) + O(\varepsilon^{\gamma \alpha}) \Bigg]
\end{split} \\
\begin{split}
= \varepsilon^{\alpha} &\Bigg[\, -\int_{\R^{N-1}}\left( \frac{1}{2}\frac{\partial}{\partial y_i}\abs{\nabla U}^2 + (U - U^p) \frac{\partial U}{\partial y_i}\right)\left(y'- \xi',0\right) Q(y') dy' \\
&  +\,\, O(\varepsilon^{\beta - \alpha}) + O(\varepsilon^{\gamma \alpha}) \Bigg]
\end{split} \\
\begin{split}
= \varepsilon^{\alpha} &\Bigg[\, -\int_{\R^{N-1}}\frac{\partial}{\partial y_i}\left( \frac{1}{2}\abs{\nabla U}^2 + \frac{1}{2}U^2 - \frac{1}{p+1}U^{p+1}\right)\left(y'- \xi',0\right) Q(y') dy' \\
&  +\,\, O(\varepsilon^{\beta - \alpha}) + O(\varepsilon^{\gamma \alpha}) \Bigg]
\end{split}
\end{align*}
and integrating by parts once more we obtain
\begin{align} \label{eq57}
\partial_{\xi_i} \Phi_{\varepsilon}(\xi) =
 \varepsilon^{\alpha}& \Bigg[\, \int_{\R^{N-1}}\left(\frac{1}{2}\abs{\nabla U}^2 + \frac{1}{2}U^2 - \frac{1}{p+1}U^{p+1} \right)(y',0)\frac{\partial Q}{\partial y_i}(y' + \xi') d y' \notag \\
 & +\,\, O(\varepsilon^{\beta - \alpha}) + O(\varepsilon^{\gamma \alpha}) \Bigg].
\end{align}
It is important to observe that the terms of order $O(\varepsilon^{\beta - \alpha})$ and $O(\varepsilon^{\gamma \alpha})$ go to zero uniformly with respect to $\xi' \in B_R(0)$.

Let $\xi_0' \in \Xi$. By definition of stable zero there exists
$\xi_{\varepsilon}' \rightarrow \xi_0'$ such that $\nabla
\Phi_{\varepsilon}(\xi_{\varepsilon}') = 0$ and consequently
$U_{\xi_{\varepsilon}} + w(\varepsilon, \xi_{\varepsilon})$ is a
solution of \eqref{eq6}, where $\xi_{\varepsilon}=
(\xi_{\varepsilon}', \frac{1}{\varepsilon}\psi(\varepsilon
\xi_{\varepsilon}')) \in \partial \Omega_{\varepsilon}$.

To finish the proof we have to show that two different stable zeroes generate two different solutions. Let $\xi_1 '$ and $\xi_2 '$ be two different stable zeroes of $\mathcal{L}$ and let $u_{1,\varepsilon}$ and $u_{2,\varepsilon}$ be the solutions generated by $\xi_1 '$ and $\xi_2 '$ respectively.

Using elliptic estimates, one can prove that the error term $w(\varepsilon,\xi_{\varepsilon})$ satisfies
\begin{equation}
\label{eq58}
\norm{w(\varepsilon,\xi_{\varepsilon})}_{L^{\infty}(\Omega_{\varepsilon})} \longrightarrow 0 \quad \text{as} \quad \varepsilon \to 0.
\end{equation}

We have
\begin{align*}
u_{1,\varepsilon} = U_{\xi_{\varepsilon}^1} + w(\varepsilon, \xi_{\varepsilon}^1)  \\
u_{2,\varepsilon} = U_{\xi_{\varepsilon}^2} + w(\varepsilon, \xi_{\varepsilon}^2)
\end{align*}
\noindent and $\xi_{\varepsilon}^1 \rightarrow (\xi_1' , 0)$, $\xi_{\varepsilon}^2 \rightarrow (\xi_2' , 0)$ as $\varepsilon \rightarrow 0$.

By \eqref{eq58}
\begin{align*}
&u_{1,\varepsilon}(\xi_{\varepsilon}^1) \longrightarrow U(0)  \\
&u_{2,\varepsilon}(\xi_{\varepsilon}^1) \longrightarrow U(\xi_1 ' - \xi_2 ' , 0) \neq U(0)
\end{align*}
\noindent which completes the proof.
\end{proof}

\begin{remark} \label{re10}
It is easy to prove that these solutions $U_{\xi_{\varepsilon}} + w(\varepsilon,\xi_{\varepsilon})$ are boundary single peak solutions, i.e., they have exactly one local maximum point in $\overline{\Omega}_{\varepsilon}$ which lies on  $\partial \Omega_{\varepsilon}$. See \cite[Lemma 4.2]{YYL} or the arguments used in \cite{NT1}.
\end{remark}

\section{Properties of boundary single peak solutions}\label{s5}

We start this section by deriving some Pohozaev-type identities
that will be useful later.

\begin{lemma} \label{le9}
Let $\phi \in C^2(\overline{\mathcal{D}})$ be a solution of
\begin{align}
\label{eq118}
\begin{cases}
  -\Delta \phi + \phi = \phi^p,  & \mathcal{D}\\
  \phi > 0,                      & \mathcal{D}\\
  \frac{\partial \phi}{\partial \nu} = 0,& \partial \mathcal{D}
\end{cases}
\end{align}
where $\nu = (\nu_1,\dots,\nu_N)$ is the unit outer normal vector field on $\partial \mathcal{D}$. Then
\begin{align}
&\int_{\partial \mathcal{D}} \left(\frac{\abs{\nabla \phi}^2}{2} + \frac{\phi^2}{2}  - \frac{\phi^{p+1}}{p+1} \right)\nu_j \,\,d \sigma =0 \label{eq101} \\
&\int_{\partial \mathcal{D}} \left(\frac{\abs{\nabla \phi}^2}{2} + \frac{\phi^2}{2}  - \frac{\phi^{p+1}}{p+1} \right)y_j \nu_i  \,d\,\sigma =
\int_{\partial \mathcal{D}} \left(\frac{\abs{\nabla \phi}^2}{2} + \frac{\phi^2}{2}  - \frac{\phi^{p+1}}{p+1} \right)y_i \nu_j \,d\,\sigma   \label{eq103}
\end{align}
for all $i,j \in \left\{1,\dots,N \right\}$.
\end{lemma}

\begin{proof}
Testing \eqref{eq118} with the function $\frac{\partial \phi}{\partial y_j}$ we obtain \eqref{eq101}, and testing \eqref{eq118} with $\frac{\partial \phi}{\partial y_i} y_j$ where $i \neq j$, we have
\begin{align*}
\int_{\mathcal{D}} \left(\nabla  \phi \cdot \nabla \left(\frac{\partial \phi}{\partial y_i}y_j\right) + \phi \frac{\partial \phi}{\partial y_i}y_j  - \phi^p \frac{\partial \phi}{\partial y_i}y_j \right) \,d y =0.
\end{align*}
Since $\nabla \left(\frac{\partial \phi}{\partial y_i}y_j\right) = \nabla \left(\frac{\partial \phi}{\partial y_i}\right)y_j +\frac{\partial \phi}{\partial y_i} e_j$, where $e_j=(0,\dots,1,\dots,0)$, we get
\begin{align*}
\int_{\mathcal{D}} \left(\nabla  \phi \cdot \nabla \left(\frac{\partial \phi}{\partial y_i}\right) + \phi \frac{\partial \phi}{\partial y_i}  - \phi^p \frac{\partial \phi}{\partial y_i} \right)y_j \,d y  + \int_{\mathcal{D}} \frac{\partial \phi}{\partial y_i} \frac{\partial \phi}{\partial y_j} dy =0 \\
\int_{\mathcal{D}} \frac{\partial}{\partial y_i}\left(\frac{\abs{\nabla \phi}^2}{2} + \frac{\phi^2}{2}  - \frac{\phi^{p+1}}{p+1} \right)y_j \,\,dy = - \int_{\mathcal{D}} \frac{\partial \phi}{\partial y_i} \frac{\partial \phi}{\partial y_j} dy  \\
\int_{\partial \mathcal{D}} \left(\frac{\abs{\nabla \phi}^2}{2} + \frac{\phi^2}{2}  - \frac{\phi^{p+1}}{p+1} \right)y_j \nu_i  \,d\,\sigma = - \int_{\mathcal{D}} \frac{\partial \phi}{\partial y_i} \frac{\partial \phi}{\partial y_j} dy
\end{align*}
which proves \eqref{eq103}.
\end{proof}

\begin{proposition}
\label{prop11}
Let $u_{\varepsilon}$ be a family of boundary single peak solutions of \eqref{eq1} concentrating at $0$ and
let $v_{\varepsilon}(x) = u_{\varepsilon}(\varepsilon x)$ be the corresponding solution of \eqref{eq6}. Then
\begin{itemize}
\item[(i)] $v_{\varepsilon}(x) = U(x-\frac{P_{\varepsilon}}{\varepsilon}) + \omega_{\varepsilon}(x)$, where
$\norm{\omega_{\varepsilon}}_{C^1(\overline{\Omega}_{\varepsilon})} \rightarrow 0$ as $\varepsilon \rightarrow 0$,
\item[(ii)] $u_{\varepsilon}(x) \leqslant C \exp\left(-\lambda \frac{\abs{x-P_{\varepsilon}}}{\varepsilon}\right)$,  for some $\lambda > 0$,
\item[(iii)] $\left| \nabla u_{\varepsilon}(x) \right| \leqslant \frac{C}{\varepsilon}
\exp\left(-\lambda
\frac{\abs{x-P_{\varepsilon}}}{\varepsilon}\right)$.
\end{itemize}
If we suppose also that the domain satisfies the conditions \eqref{eq2} - \eqref{eq97} then we have the crucial estimate
\begin{equation}
\label{eq97b}
\abs{P_{\varepsilon}} =O(\varepsilon).
\end{equation}
\end{proposition}

\begin{proof}
(i) follows by \cite[Theorem 6.1]{WW}, see also \cite{W}. To prove that ${\norm{\omega_{\varepsilon}}_{C^1(\overline{\Omega}_{\varepsilon})} \to 0}$ we use the exponential estimates (ii) and (iii).

(ii) and (iii) follow by \cite[Lemma 2.1]{LZ}.

Proof of \eqref{eq97b}.  First of all let us note that to prove \eqref{eq97b} it suffices to show that $\abs{\xi'_{\varepsilon}} = O(\varepsilon)$ because $\abs{\psi(\xi'_{\varepsilon})} \leqslant C \abs{\xi'_{\varepsilon}}^{\alpha + 1}$. By contradiction, assume that there exists a sequence $\varepsilon_n \rightarrow 0$ such that
\begin{equation}
\label{eq98}
\frac{\abs{\xi'_{\varepsilon_n}}}{\varepsilon_n} \longrightarrow \infty \qquad where \quad P_{\varepsilon_n} = (\xi'_{\varepsilon_n}, \psi(\xi'_{\varepsilon_n})).
\end{equation}
Setting $\xi'_n = \xi'_{\varepsilon_n}$ and $\widetilde{v}_n(y)= u_{\varepsilon_n}(\varepsilon_n y + P_n)$, then $\widetilde{v}_n$ satisfies
\begin{align}
\label{eq99}
\begin{cases}
  -\Delta \widetilde{v}_n + \widetilde{v}_n = \widetilde{v}_n^p,  &\widetilde{\Omega}_n \\
  \frac{\partial \widetilde{v}_n}{\partial \nu} = 0,& \partial \widetilde{\Omega}_n
\end{cases}
\end{align}
where $\widetilde{\Omega}_n = \Omega_{\varepsilon_n} - \frac{P_n}{\varepsilon_n}$. We have $0 \in \partial \widetilde{\Omega}_n$ $\forall \,\, n$ and $\widetilde{v}_n(0) = \max\limits_{\widetilde{\Omega}_n} \widetilde{v}_n$.


We write $\partial \widetilde{\Omega}_n=\left\{\left(y', \frac{1}{\varepsilon_n}\psi(\varepsilon_n y' + \xi'_n) - \frac{1}{\varepsilon_n} \psi(\xi'_n) \right)\left| \abs{y'} < \frac{r_0}{2 \varepsilon_n} \right.\right\}$ in a neighborhood of $0$.

By (ii) and (iii) we have
\begin{equation}
\label{eq100}
\widetilde{v}_n(y) + \abs{\nabla \widetilde{v}_n(y)} \leqslant C \text{e}^{-\lambda \abs{y}}.
\end{equation}

In a neighborhood of $0$ on $\partial \widetilde{\Omega}_n$ we have
\begin{equation}
\label{eq104}
\nu(y) = \frac{1}{\sqrt{1+ \abs{\nabla \psi(\varepsilon_n y' + \xi'_n)}^2}}(\nabla \psi(\varepsilon_n y' + \xi'_n), -1).
\end{equation}
Using \eqref{eq100} and \eqref{eq104}, we have that \eqref{eq101} for $\widetilde{v}_n$ becomes
\begin{equation}
\begin{split}
& \int_{\abs{y'} < \frac{r_0}{2 \varepsilon_n}} \left(\frac{\abs{\nabla \widetilde{v}_n}^2}{2}+
\frac{\widetilde{v}_n^2}{2}  - \frac{\widetilde{v}_n^{p+1}}{p+1}\right) \Big(y', \frac{1}
{\varepsilon_n} \big(\psi(\varepsilon_n y' + \xi'_n) - \psi(\xi'_n)\big) \Big) \\
& \frac{\partial \psi}{\partial y_j}(\varepsilon_n y' + \xi'_n) dy' = O\left(e^{- \frac{\tau_0}
{\varepsilon_n}} \right).
\end{split} \label{eq105}
\end{equation}
By Taylor's Theorem we have
\begin{align}
 \frac{\partial \psi}{\partial y_j}(\varepsilon_n y' + \xi'_n) =&  \frac{\partial \psi}{\partial y_j}(\xi'_n)+ \nabla \frac{\partial \psi}{\partial y_j}(\xi'_n) \cdot \varepsilon_n y' + \frac{1}{2} D^2  \frac{\partial \psi}{\partial y_j}(\xi'_n)[\varepsilon_n y', \varepsilon_n y'] \notag \\
 &+ O\left(\varepsilon_n^3 \abs{\xi'_n}^{\alpha-3} \abs{y'}^3 \right) \label{eq106}
\end{align}
which substituted into equation \eqref{eq105} gives
\begin{equation}
\label{eq107}
I_{1,n} + I_{2,n} + I_{3,n} = O(\varepsilon_n^3 \abs{\xi'_n}^{\alpha-3}).
\end{equation}
By \eqref{eq101} for $\widetilde{v}_n$ with $j=N$ and by \eqref{eq100} it follows that

\begin{equation}
\begin{split}
 I_{1,n} &=\frac{\partial \psi}{\partial
y_j}(\xi'_n)\int_{\abs{y'} < \frac{r_0}{2 \varepsilon_n}}
\left(\frac{\abs{\nabla \widetilde{v}_n}^2}{2}+ \frac{\widetilde{v}_n^2}{2}
 - \frac{\widetilde{v}_n^{p+1}}{p+1}\right)
\Big(y', \frac{1}
{\varepsilon_n} \big(\psi(\varepsilon_n y' + \xi'_n) - \psi(\xi'_n)\big) \Big)dy'\\
&=\frac{\partial \psi}{\partial y_j}(\xi'_n)\int_{\partial\widetilde{\Omega}_n\setminus\left\{\abs{y'}<
\frac{r_0}{2 \varepsilon_n}\right\}} \left(\frac{\abs{\nabla \widetilde{v}_n}^2}{2}+ \frac{\widetilde{v}_n^2}{2}
 - \frac{\widetilde{v}_n^{p+1}}{p+1}\right) \nu_N d\sigma\\
& =O\left(e^{- \frac{\tau_0} {\varepsilon_n}} \right).
\label{eq108}
\end{split}
\end{equation}
We claim that
\begin{equation}
\label{eq109}
I_{2,n} = o(\abs{\xi'_n}^{\alpha-1} \varepsilon_n^2 ).
\end{equation}
Let us assume for a moment that \eqref{eq109} holds. Then the assumption \eqref{eq98} implies that
\begin{equation}
\label{eq109b}
I_{1,n} + I_{2,n} + I_{3,n}=
o(\varepsilon_n^2\abs{\xi'_n}^{\alpha-2}),
\end{equation}
and then
\begin{equation}
\begin{split}
&\int_{\abs{y'} < \frac{r_0}{2 \varepsilon_n}} \left(\frac{\abs{\nabla \widetilde{v}_n}^2}{2} + \frac{\widetilde{v}_n^2}{2} - \frac{\widetilde{v}_n^{p+1}}{p+1} \right) \Big(y', \frac{1}
{\varepsilon_n} \big(\psi(\varepsilon_n y' + \xi'_n) - \psi(\xi'_n)\big) \Big)  \\
& D^2  \frac{\partial Q}{\partial y_j}(\xi'_n)[\varepsilon_n y',
\varepsilon_n y']dy' = o\left(\varepsilon_n^2
\abs{\xi'_n}^{\alpha-2}\right)
\end{split}\label{eq110}
\end{equation}
where we have used \eqref{eq2}-\eqref{eq4}.

Finally we rewrite \eqref{eq110} as
\begin{equation*}
\begin{split}
&\int_{\abs{y'} < \frac{r_0}{2 \varepsilon_n}} \left(\frac{\abs{\nabla \widetilde{v}_n}^2}{2} + \frac{\widetilde{v}_n^2}{2}  - \frac{\widetilde{v}_n^{p+1}}{p+1} \right) \Big(y', \frac{1}
{\varepsilon_n} \big(\psi(\varepsilon_n y' + \xi'_n) - \psi(\xi'_n)\big) \Big) \\
&D^2  \frac{\partial Q}{\partial
y_j}\left(\frac{\xi'_n}{\abs{\xi'_n}}\right)[y',y'] dy'=o(1)
\end{split}
\end{equation*}
and passing to the limit we get
\begin{equation}
\label{eq111}
\int_{\R^{N-1}} \left(\frac{1}{2}\abs{\nabla U}^2 + \frac{1}{2} U^2 - \frac{1}{p+1}\,U^{p+1}\right) \left(y',0\right)D^2  \frac{\partial Q}{\partial y_j}(\zeta)[y',y'] dy' =0
\end{equation}
where
\begin{equation*}
\lim_{n \rightarrow \infty} \frac{\xi'_n}{\abs{\xi'_n}}= \zeta \in
S^{N-2}.
\end{equation*}
Moreover, since for $k=1,\dots,N-1$,
\begin{align*}
\int_{\R^{N-1}} \left(\frac{1}{2}\abs{\nabla U}^2 + \frac{1}{2} U^2 - \frac{1}{p+1}\,U^{p+1}\right) \left(y',0\right)y_k^2 dy' =&\\
 \frac{1}{N-1}\int_{\R^{N-1}} \left(\frac{1}{2}\abs{\nabla U}^2 + \frac{1}{2} U^2 - \frac{1}{p+1}
 \,U^{p+1}\right) \left(y',0\right)\abs{y'}^2 dy'& > 0,
\end{align*}
we get from \eqref{eq111} that
\begin{equation*}
\Delta \frac{\partial Q}{\partial y_j}(\zeta) =0,\,\, \text{for}
\,\,\,\, j= 1,\dots,N-1,
\end{equation*}
which gives a contradiction to \eqref{eq97}.

It remains to prove \eqref{eq109}. Set
\begin{equation*}
a_{n,k} = \int_{\abs{y'} < \frac{r_0}{2 \varepsilon_n}} \left(\frac{\abs{\nabla \widetilde{v}_n}^2}{2} + \frac{\widetilde{v}_n^2}{2}  - \frac{\widetilde{v}_n^{p+1}}{p+1} \right) \Big(y', \frac{1}
{\varepsilon_n} \big(\psi(\varepsilon_n y' + \xi'_n) - \psi(\xi'_n)\big) \Big)y_k dy'
\end{equation*}
for $k=1,\dots,N-1$, thus
\begin{equation}
\label{eq112} I_{2,n}= \varepsilon_n \nabla \frac{\partial
\psi}{\partial y_j}(\xi'_n) \cdot a_n
\end{equation}
and by \eqref{eq112}, to prove the claim it suffices to prove that
$\abs{a_n} = o(\varepsilon_n)$. In order to do this, we use
\eqref{eq103} for $\widetilde{v}_n$ with $i=N$ and $j=k \in \left\{1,\dots,N-1 \right\}$
to obtain
\begin{align}
- a_{n,k} = &\int_{\abs{y'} < \frac{r_0}{2 \varepsilon_n}} \left(\frac{\abs{\nabla \widetilde{v}_n}^2}{2} + \frac{\widetilde{v}_n^2}{2}  - \frac{\widetilde{v}_n^{p+1}}{p+1}\right) \Big(y', \frac{1}{\varepsilon_n} \big(\psi(\varepsilon_n y' + \xi'_n) - \psi(\xi'_n)\big) \Big) \notag\\
&\left[ \frac{1}{\varepsilon_n} \psi(\varepsilon_n y' + \xi'_n) - \frac{1}{\varepsilon_n} \psi(\xi'_n)\right] \frac{\partial \psi}{\partial y_k}(\varepsilon_n y'+ \xi'_n) dy' +\, O\left(e^{- \frac{\tau_0}{\varepsilon_n}} \right)\label{eq113}.
\end{align}
By Taylor's Theorem
\begin{align}
\begin {split}
a_{n,k} =  - &\int_{\abs{y'} < \frac{r_0}{2 \varepsilon_n}} \left(\frac{\abs{\nabla \widetilde{v}_n}^2}{2} + \frac{\widetilde{v}_n^2}{2}  - \frac{\widetilde{v}_n^{p+1}}{p+1} \right) \Big(y', \frac{1}{\varepsilon_n} \big(\psi(\varepsilon_n y' + \xi'_n) - \psi(\xi'_n)\big) \Big) \notag\\
&\left[\nabla \psi(\xi'_n) \cdot y' + o(\varepsilon_n)\right] \left[\frac{\partial \psi}{\partial y_k}( \xi'_n) + o(\varepsilon_n)\right] dy' +\, O\left(e^{- \frac{\tau_0}{\varepsilon_n}} \right)
\end{split} \notag \\
\begin {split}
 = - &\int_{\abs{y'} < \frac{r_0}{2 \varepsilon_n}} \left(\frac{\abs{\nabla \widetilde{v}_n}^2}{2} + \frac{\widetilde{v}_n^2}{2}  - \frac{\widetilde{v}_n^{p+1}}{p+1} \right) \Big(y', \frac{1}{\varepsilon_n} \big(\psi(\varepsilon_n y' + \xi'_n) - \psi(\xi'_n)\big) \Big)\notag \\
&\frac{\partial \psi}{\partial y_k}( \xi'_n) \nabla \psi (\xi'_n) \cdot y' dy'+ o(\varepsilon_n)
\end{split} \notag\\
 = & \,\,\,\,T_{n,k} \cdot a_n + o(\varepsilon_n)
 \label{eq114}
\end{align}
where $T_{n,k} = - \frac{\partial \psi}{\partial y_k} (\xi'_n)\nabla \psi(\xi'_n)$ and $\abs{T_{n,k}} \longrightarrow 0$ as $n \rightarrow \infty$ for $k=1,\dots,N-1$.
By \eqref{eq114} we infer
\begin{align*}
&\abs{a_n} \leqslant \sum_{k=1}^{N-1}\abs{T_{n,k}} \abs{a_n} + o(\varepsilon_n)\notag \\
&\abs{a_n} \leqslant \frac{o(\varepsilon_n)}{1 -
\sum_{k=1}^{N-1}\abs{T_{n,k}} } = o(\varepsilon_n).
\end{align*}
The proof is now complete.
\end{proof}

\begin{remark}
In general, the assumption \eqref{eq97} can not be completely
removed. Otherwise \eqref{eq97b} might not hold. Here we give an
example.

Let $\Omega = B_1(0)$ in $\R^2$ and $u_{\varepsilon}$ be a family
of boundary single peak solutions of \eqref{eq1} with peak
$P_{\varepsilon}$ and $P_{\varepsilon} \rightarrow P_0$; we can
define $\widetilde{u}_{\varepsilon}(x) =
u_{\varepsilon}(\Theta_{\varepsilon}x)$ where
$\Theta_{\varepsilon}: \R^2 \rightarrow \R^2$ is a rotation by an
angle $\theta_{\varepsilon}$, and $\theta_{\varepsilon} \to 0$ as
$\varepsilon \rightarrow 0$.

Then $\widetilde{u}_{\varepsilon}$ is also a family of boundary
single peak solutions of \eqref{eq1} with peak
$\widetilde{P}_{\varepsilon}=
\Theta_{\varepsilon}^{-1}P_{\varepsilon}\longrightarrow P_0$. But we
can choose $\theta_{\varepsilon} \rightarrow 0$ in such a way that
$\frac{\abs{\widetilde{P}_{\varepsilon} - P_0}}{\varepsilon}
\longrightarrow \infty$.
\end{remark}

\begin{proposition}
\label{prop12}
Let $u_{\varepsilon}$ be a family of boundary single peak solutions of \eqref{eq1} and let $\varepsilon_n$ be a sequence which goes to zero. Suppose that the domain satisfies conditions \eqref{eq2} - \eqref{eq97}. Then, up to a subsequence, if $P_n$ denotes the peak of the solution $u_{\varepsilon_n}$ we have
\begin{equation}
\label{eq70}
\frac{P_n}{\varepsilon_n} \longrightarrow (\xi',0)
\end{equation}
where $\xi' \in \R^{N-1}$ satisfies $\mathcal{L}(\xi')=0$.
\end{proposition}

\begin{proof}
Let us write ${P_n = (\xi_n',\psi(\xi_n'))}$, it follows from
\eqref{eq97b} that
${\abs{\xi_n'}=O(\varepsilon_n)}$. By \eqref{eq2} and \eqref{eq4}
we have $\abs{\psi(\xi'_{\varepsilon_n})} \leqslant C
\abs{\xi'_{\varepsilon_n}}^{\alpha + 1}$ so, up to a
subsequence, we have \eqref{eq70}.

Let $v_n(y)=u_n(\varepsilon_n y)$ be the corresponding solution of \eqref{eq6}. Since $P_n \rightarrow 0$ then for $n$ sufficiently large we have $\frac{\abs{P_n}}{\varepsilon_n} < \frac{r_0}{2\varepsilon_n}$. Next we use the identity \eqref{eq101} for $v_n$
\begin{equation}
\label{eq71}
\int_{\partial \Omega_{\varepsilon_n}} \left( \frac{\abs{\nabla v_n}^2}{2} + \frac{v_n^2}{2} - \frac{v_n^{p+1}}{p+1} \right) \nu_j d\sigma = 0
\qquad j=1,\dots,N
\end{equation}
where $\nu =(\nu_j)$ is the unit outer normal field on $\partial \Omega_{\varepsilon_n}$.  Now we divide $\partial \Omega_{\varepsilon_n}$ into two parts, $\partial \Omega_{\varepsilon_n} = G_{\frac{r_0}{\varepsilon_n}} \cup \left( \partial \Omega_{\varepsilon_n} \backslash \, G_{\frac{r_0}{\varepsilon_n}}\right)$ where $G_{\frac{r_0}{\varepsilon_n}} = \left\{\left(y', \frac{1}{\varepsilon_n}\psi(\varepsilon_n y')\right)\,\,\,| \,\,\,\abs{y'} < \frac{r_0}{\varepsilon_n} \right\}$. By Proposition \ref{prop11} the integral on $\left( \partial \Omega_{\varepsilon_n} \backslash \, G_{\frac{r_0}{\varepsilon_n}}\right)$ is exponentially small, so \eqref{eq71} becomes
\begin{equation}
\int_{G_{\frac{r_0}{\varepsilon_n}}} \left( \frac{\abs{\nabla v_n}^2}{2} + \frac{v_n^2}{2} - \frac{v_n^{p+1}}{p+1} \right) \nu_j d\sigma = O(\text{e}^{\frac{-\tau_0}{\varepsilon_n}})
\end{equation}
which implies
\begin{equation}
 \int_{\abs{y'}<\frac{r_0}{\varepsilon_n}} \left( \frac{\abs{\nabla v_n}^2}{2} + \frac{v_n^2}{2} - \frac{v_n^{p+1}}{p+1} \right)\Big(y',\frac{1}{\varepsilon_n}\psi(\varepsilon_n y')\Big) \frac{\partial\psi}{\partial y_j}(\varepsilon_n y') dy' = O(\text{e}^{\frac{-\tau_0}{\varepsilon_n}}).
\end{equation}
Finally, using \eqref{eq2} - \eqref{eq4}, we have
\begin{equation}
\int_{\abs{y'}<\frac{r_0}{\varepsilon_n}} \left( \frac{\abs{\nabla v_n}^2}{2} + \frac{v_n^2}{2} - \frac{v_n^{p+1}}{p+1} \right)\Big(y',\frac{1}{\varepsilon_n}\psi(\varepsilon_n y')\Big) \frac{\partial Q}{\partial y_j}(y') dy' = O({\varepsilon_n}^{\beta - \alpha}) \label{eq72}
\end{equation}
for $j=1,\dots,N-1$. Now we use \eqref{eq70}, the properties of $v_n$ stated in Proposition \ref{prop11} and Lebesgue's Dominated Convergence Theorem to pass to the limit in \eqref{eq72}, and to get
\begin{equation}
\label{eq73}
\int_{\R^{N-1}}\left(\frac{1}{2}\abs{\nabla U_{\xi'}}^2 + \frac{1}{2}U_{\xi'}^2 - \frac{1}{p+1}U_{\xi'}^{p+1} \right)(y',0)\frac{\partial Q}{\partial y_j}(y') d y' = 0
\end{equation}
for $j=1,...,N-1$, which means that $\mathcal{L}(\xi')=0$.
\end{proof}

\section{An exact multiplicity result} \label{exact}

The aim of this section is to prove the second part of Theorem \ref{mainteo}.

\begin{proof}[Proof of \eqref{eq175}]
By \eqref{eq150} we already know that
\begin{equation}
\label{eq76}
\# \left\{\text{Single peak solutions of \eqref{eq1} concentrating at} \,\, 0 \right\} \geqslant \# \Xi.
\end{equation}
Suppose, by contradiction, that \eqref{eq76} is a strict inequality. Since $\# \Xi < \infty$, by Proposition \ref{prop12} there exists $\xi' \in \Xi$, a sequence $\varepsilon_n \to 0$ and two distinct single peak solutions $u_{1,n}$ and $u_{2,n}$ of \eqref{eq1} with $\varepsilon = \varepsilon_n$ such that if $P_{1,n}$ and $P_{2,n}$ are their peaks, we have
\begin{equation}
\label{eq77}
\lim_{n \rightarrow \infty} \frac{P_{1,n}}{\varepsilon_n} = (\xi',0) = \lim_{n \rightarrow \infty} \frac{P_{2,n}}{\varepsilon_n}.
\end{equation}
Since $u_{1,n} \not\equiv u_{2,n}$ we can consider the function
\begin{equation}
\label{eq78}
\phi_n(y) = \frac{v_{1,n}(y)- v_{2,n}(y)}{\norm{v_{1,n}- v_{2,n}}_{L^{\infty}(\Omega_{\varepsilon_n})}} \,\,\,\,\,\, y \in \overline{\Omega}_{\varepsilon_n}
\end{equation}
where $v_{1,n}(y) = u_{1,n}(\varepsilon_n y)$ and $v_{2,n}(y) = u_{2,n}(\varepsilon_n y)$ are the corresponding solutions of \eqref{eq6}. By Proposition \ref{prop11}
\begin{equation}
\norm{v_{1,n} - U_{\xi'}}_{C^1(\overline{\Omega}_{\varepsilon_n})} \to 0 \,,\,\,\norm{v_{2,n} - U_{\xi'}}_{C^1(\overline{\Omega}_{\varepsilon_n})} \to 0 \qquad \text{as}\,\,\, n \rightarrow \infty
\label{eq79}
\end{equation}
where $U_{\xi'}(\cdot) = U(\cdot - (\xi',0))$. Then $\phi_n$ satisfies
\begin{align}
\label{eq80}
\begin{cases}
  - \Delta \phi_n + \phi_n = c_n(y)\phi_n,  &\Omega_{\varepsilon_n} \\
  \frac{\partial \phi_n}{\partial \nu} = 0,& \partial \Omega_{\varepsilon_n}
\end{cases}
\end{align}
where
\begin{equation}
\label{eq81}
c_n(y) = p \int_0^1 \left( t v_{1,n}(y) + (1-t)v_{2,n}(y)\right)^{p-1}dt.
\end{equation}
Again, by Proposition \ref{prop11}, $\norm{c_n - pU_{\xi'}^{p-1}}_{C^0(\overline{\Omega}_{\varepsilon_n})} \rightarrow 0$ as $n \rightarrow \infty$. By the arguments
used in the proof of Theorem 3 of \cite{LNT} we can prove that $\phi_n \rightarrow \phi$ in $C_{loc}^2(\R_+^N)$ where $\phi$ is a bounded solution of
\begin{align}
\label{eq82}
\begin{cases}
  - \Delta \phi + \phi = pU_{\xi'}^{p-1}\phi,  &\R_+^N \\
  \frac{\partial \phi}{\partial y_N} = 0,& \partial \R_+^N \, .
\end{cases}
\end{align}
Standard arguments imply that $\phi = \sum_{i=1}^{N-1} a_i \frac{\partial U_{\xi'}}{\partial y_i}$ for some constants $a_i \in \R$. By \eqref{eq101} we have
\begin{equation*}
\int_{\partial \Omega_{\varepsilon_n}} \left( \frac{\abs{\nabla v_{1,n}}^2}{2} + \frac{v_{1,n}^2}{2} - \frac{v_{1,n}^{p+1}}{p+1} \right) \nu_j d\sigma = \int_{\partial \Omega_{\varepsilon_n}} \left( \frac{\abs{\nabla v_{2,n}}^2}{2} + \frac{v_{2,n}^2}{2} - \frac{v_{2,n}^{p+1}}{p+1} \right) \nu_j d\sigma
\end{equation*}
which we rewrite as
\begin{equation}
\label{eq84}
\int_{\partial \Omega_{\varepsilon_n}} \left( \nabla{a_n}(y)\cdot\nabla \phi_n(y) + a_n(y)\phi_n(y) - b_n(y)\phi_n(y) \right) \nu_j(y) d\sigma = 0
\end{equation}
where
\begin{align}
\label{eq85}
a_n(y) = \frac{1}{2}(v_{1,n} + v_{2,n})(y), & \,\,\, b_n(y)=\int_0^1 \left( t v_{1,n}(y) + (1-t)v_{2,n}(y)\right)^{p}dt. \\
\norm{a_n - U_{\xi'}}_{C^1(\overline{\Omega}_{\varepsilon_n})} \rightarrow 0 \label{eq86} , & \,\,\, \norm{b_n - U_{\xi'}^p}_{C^0(\overline{\Omega}_{\varepsilon_n})} \rightarrow 0
\,\,\text{  ( by Proposition \ref{prop11} ) }.
\end{align}
We want to pass to the limit in \eqref{eq84}. To do this we first observe that by definition $\abs{\phi_n} \leqslant 1$ and also $\max \{\abs{\nabla \phi_n} \,;\, {\partial \Omega_{\varepsilon_n}}\}$ is uniformly bounded (to check the last claim one can use the diffeomorphism which straightens the boundary portion near a point $P \in \Omega$, as in \cite{LNT,NT1,NT2}, then apply Schauder's estimates near the boundary, see for example \cite[Chapter 6]{GT}, to the corresponding elliptic equation satisfied in a half ball as well as the $C^{2,\alpha}$ uniform regularity of the domain). At this point, using \eqref{eq86} and Lebesgue's Dominated Convergence Theorem, we can proceed as in the proof of Proposition \ref{prop12} to pass to the limit in \eqref{eq84} and conclude that
\begin{equation}
\label{eq87}
\int_{\R^{N-1}} \left( \nabla{U_{\xi'}}\cdot\nabla \phi + U_{\xi'}\phi - U_{\xi'}^{p}\phi \right)(y',0) \frac{\partial Q}{\partial y_j}(y') dy' = 0
\end{equation}
since $\phi = \sum_{i=1}^{N-1} a_i \frac{\partial U_{\xi'}}{\partial y_i}$ we have, for any $j=1,\dots,N-1$
\begin{equation}
\label{eq88}
\sum_{i=1}^{N-1} a_i \int_{\R^{N-1}}\left[ \nabla U_{\xi'} \cdot \nabla \frac{\partial U_{\xi'}}{\partial y_i} + U_{\xi'} \frac{\partial U_{\xi'}}{\partial y_i} - U_{\xi'}^p \frac{\partial U_{\xi'}}{\partial y_i} \right](y',0) \frac{\partial Q}{\partial y_j}(y') dy' =0.
\end{equation}
From \eqref{eq174} we get that the linear system \eqref{eq88} has only the trivial solution and so $\phi \equiv 0$.

To obtain a contradiction we consider a sequence $y_n \in \overline{\Omega}_{\varepsilon_n}$ such that $\abs{\phi_n(y_n)} = \norm{\phi_n}_{L^{\infty}(\Omega_{\varepsilon_n})} = 1$. If $y_n \in \Omega_{\varepsilon_n}$ and $\phi(y_n) = 1$ ($\phi(y_n) = -1$) we have $\Delta \phi_n(y_n) \leqslant 0$ ($- \Delta \phi_n(y_n) \leqslant 0$) and by \eqref{eq80}, in any case, $c_n(y_n) \geqslant 1$. If $y_n \in \partial \Omega_{\varepsilon_n}$ we also have $c_n(y_n) \geqslant 1$ because of the boundary condition in \eqref{eq80}. If $y_n$ is bounded a contradiction arises since $\phi_n \rightarrow \phi \equiv 0$ in $C_{loc}^2(\R_+^N)$ and if $\abs{y_n} \rightarrow \infty$ we have $c_n(y_n) \rightarrow 0$ (because $\norm{c_n - U_{\xi'}}_{C^0(\overline{\Omega}_{\varepsilon_n})} \to 0$) which contradicts $c_n(y_n) \geqslant 1$ and so the theorem follows.
\end{proof}

\section{Examples and applications}\label{s7}

We would like to present some applications of the previous results.
\begin{example}
Suppose that $N=3$ and that $\psi(x_1,x_2) = x_1^5 - x_1x_2^4$. Then by direct computations we infer that
\begin{equation}
\label{eq89}
\mathcal{L}(\xi_1,\xi_2) = \left(4c_4 + 30c_2\xi_1^2 - 6c_2\xi_2^2, -12c_2\xi_1\xi_2 \right)
\end{equation}
where
\begin{equation}
 c_m =  \int_{\R^{2}}\left(\frac{1}{2}\abs{\nabla U(y',0)}^2 + \frac{1}{2}U^2(y',0) -
\frac{1}{p+1}U^{p+1}(y',0) \right)y_j^m d y'  \,.\label{eq89b}
\end{equation} 
Notice that
\begin{align}
& \text{$c_m > 0$ if $m$ is a positive even integer} \label{eq90} \\
& \text{$c_m = 0$ whenever $m=0$ or $m$ is odd.} \label{eq91}
\end{align}
We left these calculations to the Appendix \ref{apb}. In this case it is easy to check that $\Xi = \left\{\left(0,\sqrt{\frac{2c_4}{3c_2}}\right),\left(0,-\sqrt{\frac{2c_4}{3c_2}}\right) \right\}$ and that
\begin{align}
\label{eq92}
\text{Jac }\mathcal{L}(\xi_1,\xi_2) =
\begin{pmatrix}
60c_2\xi_1  & -12c_2\xi_2 \\
-12c_2\xi_2 & -12c_2\xi_1
\end{pmatrix}.
\end{align}
Clearly, by Theorem \eqref{mainteo}, in this case \eqref{eq1} admits exactly two boundary single peak solutions concentrating at $P=0$.
\end{example}

\begin{proposition}[Non existence result] \label{teo14}
Suppose that, in a neighborhood of $0$, $\Omega$ is the graph of $\psi(x') = \sum_{j=1}^{N-1} a_j x_j^{\alpha_j}$ where $a_j \in \R \setminus \{0\}$ and $\alpha_j \geqslant 3$ are positive integers. If at least one of the integers $\alpha_j$ is odd then there is no boundary single peak solution of \eqref{eq1} concentrating at $P=0$.
\end{proposition}
\begin{proof}
It suffices to look at the $j$-component of $\mathcal{L}$
\begin{align}
\mathcal{L}_j(\xi) &= \alpha_j a_j \int_{\R^{N-1}}\left(\frac{1}{2}\abs{\nabla U}^2 + \frac{1}{2}U^2 - \frac{1}{p+1}U^{p+1} \right)(y',0)(y_j + \xi_j)^{\alpha_j-1} dy \notag \\
& = \alpha_j a_j\sum_{k=0}^{\alpha_j-1} \binom{\alpha_j -1}{k}c_{k}\xi_j^{\alpha_j -1 - k} \notag
\quad(\hbox{see \eqref{eq89b} for the definition of }c_k) \\
& = \alpha_j a_j\sum_{k=1}^{\frac{\alpha_j-1}{2}} \binom{\alpha_j -1}{2k}c_{2k}\xi_j^{\alpha_j -1 - 2k} \,\,\,\,\,\,\,\,\,\,\,\text{by \eqref{eq91}}. \notag
\end{align}
Observe that all terms in the last sum have the same sign and by \eqref{eq90} we conclude that $\mathcal{L}(\xi') \neq 0$ for all $\xi' \in \R^{N-1}$. Therefore the Theorem follows from Proposition \ref{prop12}.
\end{proof}

\begin{theorem}
Suppose that $\Omega \subset \R^2$ is a $C^{\infty}$ domain and let $P_0 \in \partial \Omega$ be a point such that $H(P_0) = 0$, where $H$ denotes the mean curvature of $\partial \Omega$. Suppose also that the function $H$ has a nonzero derivative at $P_0$, then \eqref{eq1} possesses at most one boundary single peak solution concentrating at $P_0$. More precisely, if $m$ is the order of the first nonzero derivative of $H$ at $P_0$ then
\begin{enumerate}
    \item[$(i)$] If $m$ is odd then there is no boundary single peak solution of \eqref{eq1} concentrating at $P_0$;
    \item[$(ii)$] If $m$ is even then there is exactly one boundary single peak solution of \eqref{eq1} concentrating at $P_0$.
\end{enumerate}
\end{theorem}
\begin{proof}
Without loss of generality we may assume that $P_0 = 0$ and that
$\partial \Omega$, around $0$, is the graph of a function
$\psi(t)$ such that $\psi(0) = 0 = \psi'(0)$. Since the dimension
$N$ equals 2, the mean curvature is actually the curvature of a plane
curve given by the formula
\begin{equation}
\label{eq115}
H((t,\psi(t))) = H(t) = \frac{\psi''(t)}{(1+(\psi'(t))^2)^{\frac{3}{2}}}.
\end{equation}

Since $H$ has a nonzero derivative at $0$ then, by \eqref{eq115},
$\psi$ also has a nonzero derivative at $0$. If $m$ denotes the
order of the first nonzero derivative of $H$ at $0$, we claim that
$m+2$ is the order of the first nonzero derivative of $\psi$ at
$0$. Indeed, let $n$ denote the order of the first nonzero
derivative of $\psi$ at $0$ and write
\begin{equation*}
H(t) = \psi''(t)g(t) \qquad \text{where} \qquad g(t) =\frac{1}{(1+(\psi'(t))^2)^{\frac{3}{2}}}.
\end{equation*}
Now we use the Taylor expansion for $H$, $\psi''$ and $g$ to write
\begin{align*}
\frac{1}{m!}H^{(m)}(0)t^m + O(\abs{t}^{m+1}) &= \left( \frac{1}{(n-2)!}\psi^{(n)}(0)t^{n-2} + O(\abs{t}^{n-1}) \right) \left( 1 + O(\abs{t}^2)\right) \\
&= \frac{1}{(n-2)!}\psi^{(n)}(0)t^{n-2} + O(\abs{t}^{n-1})
\end{align*}
which readily implies that $n-2=m$. So we can write
\begin{equation*}
\psi(t) = \frac{1}{(m+2)!}\psi^{(m+2)}(0)t^{m+2} + O(\abs{t}^{m+3}).
\end{equation*}
By the condition $H(0) = 0$ it follows that $m \geqslant 1$. If
$m$ is odd then the result follows by Proposition \ref{teo14}, while if
$m$ is even we have
\begin{equation*}
\mathcal{L}(t) = \frac{\psi^{(m+2)}(0)}{(m+1)!} \int_{\R}\left(\frac{1}{2}\abs{\nabla U}^2 + \frac{1}{2}U^2 - \frac{1}{p+1}U^{p+1} \right)(y,0)(y + t)^{m+1} dy.
\end{equation*}
As in the proof of Theorem \ref{teo14}, by \eqref{eq90}-\eqref{eq91}, it follows that $\Xi = \{0\}$ and that $\mathcal{L}'(0) \neq 0 $. Now the result follows by Theorem \ref{mainteo}.
\end{proof}
\appendix
\section{The proof of \eqref{eq90} and \eqref{eq91}} \label{apb}

We recall the definition of $c_m$
\begin{equation*}
 c_m =  \int_{\R^{N-1}}\left(\frac{1}{2}\abs{\nabla U(y',0)}^2 + \frac{1}{2}U^2(y',0) - \frac{1}{p+1}U^{p+1}(y',0) \right)y_i^m d y' \,\,\,\, i=1,\dots,N-1.
\end{equation*}
If $m$ is odd then $c_m = 0$ since the integrand is odd in the variable $y_i$. It remains to evaluate the integral when $m=2k$, $k\in \{0,1,2,\dots\}$. First we write $v(y')=U(y',0)$ and $g(y')=\frac{\partial^2 U}{\partial y_N^2}(y',0)$. By the radial symmetry of $U$ we have
\begin{equation*}
g(y') = \sum_{j=1}^{N-1}\frac{\partial v}{\partial y_j}\frac{y_j}{\abs{y'}^2}.
\end{equation*}
Observe that $v \in H^1(\R^{N-1}) \backslash \{0\}$ satisfies the following equation
\begin{equation}
\label{eq93}
-\Delta v + v -v^p = g \,\,\, \text{ in } \R^{N-1} .
\end{equation}

Testing \eqref{eq93} with $\frac{\partial v}{\partial y_i}\frac{y_i^{2k+1}}{2k+1}$ we obtain
\begin{align}
\int_{\R^{N-1}}\left(-\Delta v + v -v^p \right)\frac{\partial v}{\partial y_i}\frac{y_i^{2k+1}}{2k+1} dy'= \int_{\R^{N-1}} g \frac{\partial v}{\partial y_i}\frac{y_i^{2k+1}}{2k+1} dy' \notag \\
\int_{\R^{N-1}} \nabla v \cdot \nabla \left(\frac{\partial v}{\partial y_i}\frac{y_i^{2k+1}}{2k+1}\right) + (v - v^p)\frac{\partial v}{\partial y_i}\frac{y_i^{2k+1}}{2k+1} =\int_{\R^{N-1}} g \frac{\partial v}{\partial y_i}\frac{y_i^{2k+1}}{2k+1} \label{eq94}
\end{align}
since
\begin{equation*}
\nabla \left(\frac{\partial v}{\partial y_i}\frac{y_i^{2k+1}}{2k+1}\right) = \nabla \left(\frac{\partial v}{\partial y_i}\right)\frac{y_i^{2k+1}}{2k+1} +  \frac{\partial v}{\partial y_i}y_i^{2k}e_i
\end{equation*}
where $e_i=(0,\dots,1,\dots,0)$, then \eqref{eq94} becomes
\begin{align}
\int_{\R^{N-1}}&\left( \nabla v \cdot \nabla \left(\frac{\partial v}{\partial y_i}\right) + v\frac{\partial v}{\partial y_i} - v^p\frac{\partial v}{\partial y_i}\right)\frac{y_i^{2k+1}}{2k+1} = \notag \\ &\int_{\R^{N-1}} g \frac{\partial v}{\partial y_i} \frac{y_i^{2k+1}}{2k+1} - \int_{\R^{N-1}}\left(\frac{\partial v}{\partial y_i} \right)^2y_i^{2k} \notag ;\\
\int_{\R^{N-1}}\frac{\partial}{\partial y_i}&\left( \frac{1}{2}\abs{\nabla v}^2 + \frac{1}{2}v^2 - \frac{1}{p+1}v^{p+1}\right)\frac{y_i^{2k+1}}{2k+1} = \notag \\
&\int_{\R^{N-1}}\sum_{j=1}^{N-1}\frac{\partial v}{\partial y_j}\frac{\partial v}{\partial y_i}\frac{y_jy_i^{2k+1}}{(2k+1)\abs{y'}^2} - \int_{\R^{N-1}}\left(\frac{\partial v}{\partial y_i}\right)^2y_i^{2k} \notag ; \\
- \int_{\R^{N-1}}&\left( \frac{1}{2}\abs{\nabla v}^2 + \frac{1}{2}v^2 - \frac{1}{p+1}v^{p+1}\right)y_i^{2k} = \notag \\
&\int_{\R^{N-1}}\sum_{j=1}^{N-1}U'(\abs{y'})^2\frac{y_j^2y_i^{2k+2}}{(2k+1)\abs{y'}^4} - \int_{\R^{N-1}}U'(\abs{y'})^2\frac{y_i^{2k+2}}{\abs{y'}^2} \notag ;
\end{align}
and then we infer
\begin{align}
&\int_{\R^{N-1}}\left( \frac{1}{2}\abs{\nabla U(y',0)}^2 +  \frac{1}{2}U^2(y',0) - \frac{1}{p+1}U^{p+1}(y',0)\right)y_i^{2k} = \notag \\
&\left( 1 - \frac{1}{2k+1}\right) \int_{\R^{N-1}} U'(\abs{y'})^2 \frac{y_i^{2k+2}}{\abs{y'}^2}  \notag
\end{align}
and the proof of \eqref{eq90} and \eqref{eq91} is now complete.
\section{The proof of \eqref{eq18}} \label{apc}
We have to prove that, for $i=1,\dots,N-1$
\begin{equation}
\label{eq116}
\left(\,\,\int_{\partial \Omega_{\varepsilon}}{ \left| \frac{\partial}{\partial \nu} \frac{\partial U_{\xi}}{\partial y_i} \right|^2  d \sigma } \right)^{\frac{1}{2}} = O(\varepsilon^{\alpha}) \qquad \text{for} \quad \xi=\left(\xi',\frac{1}{\varepsilon}\psi(\varepsilon \xi')\right) ,\,\,\, \abs{\xi'} \leqslant R .
\end{equation}
By the exponential decay of $U$ and its derivatives, it suffices to estimate the integral on $\partial \Omega_{\varepsilon} \cap B_{\frac{r_0}{\varepsilon}}(0)$.
We have
\begin{equation*}
\frac{\partial U_{\xi}}{\partial y_i}(y)= \frac{\partial U}{\partial y_i}(y-\xi) = U'(\abs{y-\xi})\frac{y_i - \xi_i}{\abs{y-\xi}}
\end{equation*}
and
\begin{equation*}
\frac{\partial^2 U_{\xi}}{\partial y_j \partial y_i}(y)=  U''(\abs{y-\xi})\frac{(y_i - \xi_i)(y_j-\xi_j)}{\abs{y-\xi}^2} + U'(\abs{y-\xi}) \left( \frac{\delta_{ij}}{\abs{y-\xi}} - \frac{(y_i - \xi_i)(y_j-\xi_j)}{\abs{y-\xi}^3} \right).
\end{equation*}
Since
\begin{equation*}
\nu(y) = \frac{(\nabla \psi (\varepsilon y'), -1)}{\sqrt{1+ \abs{\nabla \psi (\varepsilon y')}^2}}
\end{equation*}
we get
\begin{align*}
\left| \frac{\partial}{\partial \nu} \frac{\partial U_{\xi}}{\partial y_i}(y) \right| &= \left| \nabla \frac{\partial U_{\xi}}{\partial y_i}(y) \cdot \nu(y) \right| \\
&\leqslant \frac{1}{\sqrt{1+ \abs{\nabla \psi (\varepsilon y')}^2}} \left[ \left( \abs{U''(\abs{y-\xi})} + \frac{\abs{U'(\abs{y-\xi})}}{\abs{y-\xi}} \right) \abs{\nabla \psi(\varepsilon y')} \right. \\
&\qquad \qquad \qquad + \left. \left( \frac{\abs{U''(\abs{y-\xi})}}{\abs{y-\xi}} + \frac{\abs{U'(\abs{y-\xi})}}{\abs{y-\xi}^2} \right) \frac{1}{\varepsilon}\left|  \psi(\varepsilon y') -  \psi(\varepsilon \xi') \right| \right] \\
& \text{by} \,\, \eqref{eq3},\,\eqref{eq4} \,\, \text{and the Mean Value Theorem} \\
&\leqslant \frac{\varepsilon^{\alpha} (  \abs{y'}^{\alpha}  + \abs{y'}^{\beta} + R^{\beta} )   }{\sqrt{1+ \abs{\nabla \psi (\varepsilon y')}^2}} \left( \abs{U''(\abs{y-\xi})} + \frac{\abs{U'(\abs{y-\xi})}}{\abs{y-\xi}} \right).
\end{align*}
Hence,
\begin{align*}
\begin{split}
\int_{\partial \Omega_{\varepsilon} \cap B_{\frac{r_0}{\varepsilon}}(0)} \left| \frac{\partial}{\partial \nu} \frac{\partial U_{\xi}}{\partial y_i} \right|^2  d \sigma \leqslant \varepsilon^{2\alpha} \Bigg[ & \int_{\R^{N-1}}  \abs{U''(\abs{y-\xi})}^2 (  \abs{y'}^{\alpha}  + \abs{y'}^{\beta} + R^{\beta} )^2 dy' \\
&+ \int_{\R^{N-1}} \left(\frac{\abs{U'(\abs{y-\xi})}}{\abs{y'-\xi'}}\right)^2 (  \abs{y'}^{\alpha}  + \abs{y'}^{\beta} + R^{\beta} )^2 dy' \Bigg]
\end{split} \\
\begin{split}
 \leqslant C \varepsilon^{2\alpha} \Bigg[ &1 + \int_{\abs{y'-\xi'} \geqslant 1} \abs{U'(\abs{y-\xi})}^2 (  \abs{y'}^{\alpha}  + \abs{y'}^{\beta} + R^{\beta} )^2 dy'   \\
&+ \int_{\abs{y'-\xi'} \leqslant 1}  (  \abs{y'}^{\alpha}  + \abs{y'}^{\beta} + R^{\beta} )^2 dy'  \Bigg]
\end{split} \\
\leqslant  C \varepsilon^{2\alpha}
\end{align*}
and \eqref{eq116} follows.
\section{The proof of \eqref{eq5}} \label{apd}

We want to prove that, for $i=1, \dots, N-1$ and $\xi=\left(\xi',\frac{1}{\varepsilon}\psi(\varepsilon \xi')\right)$ , $\abs{\xi'} \leqslant R$
\begin{equation}
\label{eq117}
\left\|\frac{\partial U_{\xi}}{\partial y_i} - \frac{\partial U_{\xi'}}{\partial y_i} \right\|_{H^1(\Omega_{\varepsilon})} = \,\, O(\varepsilon^{\alpha})
\end{equation}
where ${U_{\xi'} = U_{(\xi',0)}}$.

By the Mean Value Theorem we can write
\begin{equation*}
\frac{\partial U_{\xi}}{\partial y_i} - \frac{\partial U_{\xi'}}{\partial y_i} =  -\frac{1}{\varepsilon}\psi(\varepsilon \xi')\frac{\partial^2 U_{\xi}}{\partial y_i \partial y_N}\left(y'-\xi',y_N - \frac{\theta}{\varepsilon}\psi(\varepsilon \xi')\right) \qquad \text{for some} \,\,\, \theta \in (0,1)
\end{equation*}
and then
\begin{align*}
\int_{\Omega_{\varepsilon}} \left|\frac{\partial U_{\xi}}{\partial y_i} - \frac{\partial U_{\xi'}}{\partial y_i} \right|^2dy
&\leqslant C \varepsilon^{2\alpha} \int_{\Omega_{\varepsilon}} \left| \frac{\partial^2 U_{\xi}}{\partial y_i \partial y_N}\left(y'-\xi',y_N - \frac{\theta}{\varepsilon}\psi(\varepsilon \xi')\right) \right|^2 dy \\
& \leqslant C \varepsilon^{2\alpha} \int_{\R^N} \exp\Bigg(- \lambda \Big(\abs{y'-\xi'} + \big|y_N - \frac{\theta}{\varepsilon}\psi(\varepsilon \xi')\big|\Big)\Bigg)dy \\
& \leqslant C \varepsilon^{2\alpha} \int_{\R^N} \exp\big(- \lambda\abs{y'-\xi'} \big) \exp \Big(- \lambda\big|y_N - \frac{\theta}{\varepsilon}\psi(\varepsilon \xi')\big|\Big)dy \\
& \leqslant C \varepsilon^{2\alpha} \text{e}^{\lambda \abs{\xi'}}\text{e}^{\lambda \abs{\frac{\psi(\varepsilon \xi')}{\varepsilon}}}  \int_{\R^N}  \text{e}^{-\lambda \abs{y}}dy \\
& \leqslant C \varepsilon^{2\alpha}.
\end{align*}
Similarly we can prove that
\begin{equation*}
\int_{\Omega_{\varepsilon}} \left|\nabla\frac{\partial U_{\xi}}{\partial y_i} - \nabla\frac{\partial U_{\xi'}}{\partial y_i} \right|^2dy \leqslant C \varepsilon^{2\alpha}
\end{equation*}
and \eqref{eq117} follows.
\begin{acknowledgements}
This work was done while S\'ergio L. N. Neves was visiting the Mathematics Department of the University of Rome ``La Sapienza'' whose members he would like to thank for their warm hospitality. S\'ergio L. N. Neves was partially supported by CAPES/Brazil 2038/10-2 and CNPq/Brazil 140934/2010-3.
\end{acknowledgements}

\bibliographystyle{amsplain}

\end{document}